\documentclass[12pt,a4paper]{amsart}
\usepackage{amsthm,amsfonts,amsmath,amssymb,latexsym}
\usepackage{epsfig,graphics,color}
\usepackage[all]{xy}
\usepackage[breaklinks=true]{hyperref}
\usepackage{mathrsfs}
\usepackage{stmaryrd}
\usepackage{verbatim}
\usepackage{bm}
\usepackage{mathabx}
\usepackage{enumitem}
\usepackage{pgf,amsmath,tikz,type1cm,fix-cm}
\usetikzlibrary{positioning}
 
\newlength{\XHeight}
\newlength{\XWidth}
\setlist[itemize,1]{leftmargin=\dimexpr 26pt-.1in}

\newtheorem{PARA}{}[section]
\newtheorem{theorem}[PARA]{Theorem}
\newtheorem{corollary}[PARA]{Corollary}
\newtheorem{lemma}[PARA]{Lemma}
\newtheorem{proposition}[PARA]{Proposition}
\newtheorem{definition}[PARA]{Definition}
\newtheorem{definition-proposition}[PARA]{Definition-Proposition}
\newtheorem{definition-lemma}[PARA]{Definition-Lemma}

\theoremstyle{definition}
\newtheorem{remark}[PARA]{Remark}
\theoremstyle{theorem}
\newtheorem{example}[PARA]{Example}
\newcommand{\para}{\begin{PARA}\rm}
\newcommand{\arap}{\end{PARA}\rm}
\newcommand{\dfn}{\begin{definition}\rm}
\newcommand{\nfd}{\end{definition}\rm}
\newcommand{\rmk}{\begin{remark}\rm}
\newcommand{\kmr}{\end{remark}\rm}
\newcommand{\xmpl}{\begin{example}\rm}
\newcommand{\lpmx}{\end{example}\rm}

\newcommand{\cI}{\mathcal{I}}

\newcommand{\cK}{\mathcal{K}}
\newcommand{\cL}{\mathcal{L}}

\newcommand{\cO}{\mathcal{O}}

\newcommand{\cP}{\mathcal{P}}

\newcommand{\cU}{\mathcal{U}}

\newcommand{\cV}{\mathcal{V}}
\newcommand{\cW}{\mathcal{W}}

\newcommand{\hatotimes}{\, {\wh\otimes}}

\newcommand{\C}{{\mathbb{C}}}

\newcommand{\N}{{\mathbb{N}}}

\newcommand{\R}{{\mathbb{R}}}

\newcommand{\Z}{{\mathbb{Z}}}
\DeclareMathOperator{\coker}{coker}
\DeclareMathOperator*{\colim}{colim}
\newcommand{\im}{\mathrm{im}\,}        
\newcommand{\id}{\mathrm{ id}}         
\newcommand{\Id}{\mathrm{ Id}}

\newcommand{\ev}{\mathrm{ev}}
\newcommand{\Ev}{\mathrm{Ev}}

\newcommand{\Hom}{\mathrm{Hom}}

\newcommand{\fin}{\mathrm{fin}}
\newcommand{\Top}{\mathrm{Top}}
\newcommand{\eps}{{\varepsilon}}
\newcommand{\om}{{\omega}}

\def\NABLA#1{{\mathop{\nabla\kern-.5ex\lower1ex\hbox{$#1$}}}}
\def\Nabla#1{\nabla\kern-.5ex{}_{#1}}
\def\Tabla#1{\Tilde\nabla\kern-.5ex{}_{#1}}
\renewcommand{\Tilde}{\widetilde}

\newcommand{\p}{{\partial}}

\newenvironment{enum}
{\begin{enumerate}}
{\end{enumerate}}

\newcommand{\wh}{\widehat}
\newcommand{\ol}{\overline}

\newcommand{\into}{\hookrightarrow}

\newcommand{\bk}{\mathbf{k}}
\definecolor{darkgreen}{rgb}{0.12, 0.3, 0.17}
\definecolor{burntorange}{rgb}{0.8, 0.33, 0.0}
\definecolor{chromeyellow}{rgb}{1.0, 0.65, 0.0}
\definecolor{darkorange}{rgb}{1.0, 0.55, 0.0}
\definecolor{flame}{rgb}{0.89, 0.35, 0.13}
\definecolor{lightgray}{rgb}{0.75, 0.75, 0.75}
\usetikzlibrary{arrows}
\usetikzlibrary{shapes}

\begin{document}

\title[Rabinowitz Floer homology as a Tate vector space]{Rabinowitz
  Floer homology as a Tate vector space} 
\author{Kai Cieliebak}
\address{Universit\"at Augsburg \newline Universit\"atsstrasse 14, D-86159 Augsburg, Germany}
\email{kai.cieliebak@math.uni-augsburg.de}
\author{Alexandru Oancea}
\address{Universit\'e de Strasbourg \newline 
Institut de recherche math\'ematique avanc\'ee, IRMA \newline
Strasbourg, France}
\email{oancea@unistra.fr}
\date{\today}


\begin{abstract} 
We show that the category of linearly topologized vector spaces over discrete fields constitutes the correct framework for algebraic structures on Floer homologies with field coefficients. Our case in point is the Poincaré duality theorem for Rabinowitz Floer homology. We prove that Rabinowitz Floer homology is a locally linearly compact vector space in the sense of Lefschetz, or, equivalently, a Tate vector space in the sense of Beilinson-Feigin-Mazur. Poincaré duality and the graded Frobenius algebra structure on Rabinowitz Floer homology then hold in the topological sense. Along the way, we develop in a largely self-contained manner the theory of linearly topologized vector spaces, with special emphasis on duality and completed tensor products, complementing results of Beilinson-Drinfeld, Beilinson, Rojas, Positselski, and Esposito-Penkov.
\end{abstract}

\maketitle

\tableofcontents

\section{Introduction}\label{sec:intro}

Rabinowitz Floer homology is a Floer
theory akin to symplectic homology~\cite{Cieliebak-Frauenfelder,Cieliebak-Frauenfelder-Oancea} which associates to each Liouville domain $V$ and field $\bk$ a
graded $\bk$-vector space $RFH_*(\p V)$. Liouville domains are central objects of study in symplectic topology, as they include for example Stein domains and unit disc bundles in phase spaces. 

A key feature of Rabinowitz Floer homology is that it satisfies a\break {\em
  Poincar\'e duality} isomorphism with Rabinowitz Floer cohomology~\cite{CHO-PD}, namely
\begin{equation}\label{eq:PD}
   PD:RFH_*(\p V) \stackrel\simeq\longrightarrow RFH^{1-*}(\p V).
\end{equation}
This isomorphism intertwines natural graded Frobenius algebra structures on both sides~\cite{CHO-PD}.  
Given that $RFH_*(\p V)$ and $RFH^*(\p V)$ are typically infinite dimensional\footnote{
They are infinite dimensional in all known examples where they are nonzero.}, 
the question arises whether this isomorphism is compatible with an interpretation of Rabinowitz Floer cohomology as the dual of Rabinowitz Floer homology. In many examples (e.g.~if $V=D^*M$ is the unit disk cotangent bundle
of a simply connected closed manifold $M$), $RFH_*(\p V)$ is finite
dimensional in each degree and~\eqref{eq:PD} makes sense because 
$RFH^*(\p V)$ is the {\em degree-wise dual}.  
This fails for $V=D^*M$ as soon as $\pi_1(M)$ has infinitely many
conjugacy classes (e.g.~for $M=S^1$). 
The point of view originally taken in~\cite{CHO-PD} is to remember the action filtrations on both sides and
interpret~\eqref{eq:PD} in the filtered sense. 

In this paper, we offer a different solution which we find both simple
and satisfactory, and which can be applied to other infinite dimensional contexts in symplectic topology, particularly in relation to Calabi-Yau structures, e.g.~\cite{BJK,GGV}. 
We will see that Rabinowitz Floer homology and
cohomology carry naturally the structure of {\em topological} vector
spaces such that $RFH^*(\p V)$ is the {\em topological dual} of $RFH_*(\p V)$. 
In fact, they belong to a particular class of spaces which we now describe. 

{\bf Tate vector spaces.} We equip the field $\bk$ with the discrete topology. A {\em linearly
  topologized $\bk$-vector space} $A$ is a topological $\bk$-vector space 
which is Hausdorff and has a neighborhood basis of $0$ consisting of
linear subspaces. It is called {\em discrete} if $\{0\}$ is open, and
{\em linearly compact} if it is complete and $A/U$ is finite
dimensional for each open linear subspace $U\subset A$.
A {\em Tate vector space} $A$ is a linearly topologized vector space
which possesses an open linearly compact subspace. 
Tate spaces were introduced by Lefschetz~\cite[II.27]{Lefschetz-book} under the
name ``locally linearly compact vector spaces'', and later named
``Tate vector spaces'' in recognition of Tate's work~\cite{Tate68}.
They enjoy nice duality properties: if $A$ is a Tate vector space,
then so is its topological dual $A^*$ and the canonical map $A\to
A^{**}$ is a topological isomorphism. 

To phrase our first result, we define a {\em Tate$^\omega$ vector space} to be a linearly topologized vector space which possesses an open linearly compact subspace with a countable basis of linear neighborhoods of $0$ and of countable codimension.

\begin{theorem}\label{thm:main1}
Rabinowitz Floer homology and cohomology of any Liouville domain $V$ are
Tate$^\omega$ vector spaces such that $RFH^*(\p V)$ 
is the topological dual of
$RFH_*(\p V)$, and~\eqref{eq:PD} is a topological isomorphism. 
\end{theorem}

Thus Rabinowitz Floer homology belongs to the class of Tate vector
spaces that are {\em self-dual}. It was observed
in~\cite{Esposito-Penkov23} that an $\infty$-dimensional Tate vector
space $A$ which is self-dual has the form $A\cong D\oplus D^*$ for an
$\infty$-dimensional discrete vector space $D$. This reflects on the
algebraic level the splitting theorem from~\cite{CHO-PD} (essentially
proved in~\cite{Cieliebak-Frauenfelder-Oancea}): Rabinowitz Floer
homology equals the direct sum of reduced symplectic homology and
cohomology, 
\begin{equation} \label{eq:RFH-SH-bar}
   RFH_*(\p V) \cong \ol{SH}_*(V)\oplus \ol{SH}^{1-*}(V). 
\end{equation}
Rabinowitz Floer homology is defined via its action truncated versions
as $RFH_*(\p V) = \colim_b
\lim_a \, RFH_*^{(a,b)}(\p V)$, where $a\to-\infty$ and $b\to+\infty$.  
In general, limits and colimits in a bidirected system do not commute, see e.g.~\cite{CF}.
It is a crucial ingredient in the proof of Theorem~\ref{thm:main1} 
that in this case they do:

\begin{theorem}\label{thm:main2}
For any Liouville domain $V$, the canonical homomorphism
$$
  \kappa\colon 
  \colim_b \lim_a \, RFH_*^{(a,b)}(\p V)
  \to \lim_a \colim_b \, RFH_*^{(a,b)}(\p V)
$$
is an isomorphism of Tate$^\omega$ vector spaces.
\end{theorem}

{\bf Tensor products. }
Next, we turn to the graded Frobenius algebra structure on $RFH_*(\p
V)$ defined in~\cite{CHO-algebra}. 
It is encoded in the unital pair-of-pants product $\mu$ and the
secondary pair-of-pants coproduct $\lambda$. Simple examples (e.g.~the
unit cotangent bundle of $S^1$) show that in general $\lambda$ does
not take values in $RFH_*(\p V)\otimes RFH_*(\p V)$ but only in a
certain completion. So the question arises on which tensor products
$\mu$ and $\lambda$ are defined, and whether it can be done so that
Poincar\'e duality intertwines $\mu,\lambda$ with their topological
duals $\lambda^\vee,\mu^\vee$. 
The answer is provided by Beilinson~\cite{Beilinson}, who defines two
natural topologies, called * and !, on the algebraic tensor product $A\otimes B$ of two linearly
topologized vector spaces. The completions $A\hatotimes^*B$ and
$A\hatotimes^!B$ of $A\otimes B$ with respect to these two topologies
have the following property (Proposition~\ref{prop:duality_intertwines}):
if $A,B$ are Tate, then
$$
  (A\hatotimes^*B)^* = A^*\hatotimes^!B^* \quad\text{and}\quad
  (A\hatotimes^!B)^* = A^*\hatotimes^*B^*.
$$
(However, the category of Tate vector spaces is not closed under the
completed tensor products.)

\begin{theorem}\label{thm:main3}
For any Liouville domain $V$, the pair-of-pants product $\mu$ and the
secondary pair-of-pants coproduct $\lambda$ on Rabinowitz Floer
homology $A=RFH_*(\p V)$ define continuous linear maps
\begin{gather*}
  \mu: A\hatotimes^*A\to A \quad\text{and}\quad
  \lambda: A\to A\hatotimes^!A. 
\end{gather*}
The Poincar\'e duality isomorphism~\eqref{eq:PD} intertwines them with
their topological duals on $A^*=RFH^*(\p V)$,  
\begin{gather*}
  \lambda^\vee: A^*\hatotimes^*A^* = (A\hatotimes^!A)^* \to A^*, \cr
  \mu^\vee: A^*\to (A\hatotimes^*A)^* = A^*\hatotimes^!A^*.
\end{gather*}
\end{theorem}

As shown in~\cite{CHO-algebra}, this result implies that Rabinowitz Floer homology is a graded
Frobenius algebra in the category of Tate vector spaces.

\begin{remark}
The above notions on linearly topologized vector spaces have analogues
for Banach spaces (over $\R$ or $\C$): Tate vector spaces correspond
to reflexive Banach spaces, self-dual Tate vector spaces to Hilbert
spaces (well, almost), $\otimes^*$ to the projective tensor product,
and $\otimes^!$ to the injective tensor product (see
Remark~\ref{prop:*topological}).  
\end{remark}

{\bf Organization of the paper.} In Sections~\ref{sec:lin-top}--\ref{sec:count} we develop in a
largely self-contained manner the theory of linearly topologized
vector spaces, with special emphasis on $\Hom$ spaces, duality and completed tensor
products. The theorems stated above are proved in Sections~\ref{sec:RFH} and~\ref{sec:RFHprodcoprod}. 

Many of our results in Sections~\ref{sec:lin-top}--\ref{sec:count} are either new, or complement to maximal generality the literature. While Section~\ref{sec:lin-top} consists mostly of recollections, already in Section~\ref{sec:Tate} we distance ourselves from the existing literature by simultaneously presenting the Lefschetz definition and the Beilinson-Feigin-Mazur definition of Tate vector spaces and proving their equivalence (Proposition~\ref{prop:equivalence-Tate-Beilinson}). We also prove in Proposition~\ref{prop:self-duality} a self-duality result that generalizes Esposito-Penkov~\cite[Lemma 2.5(iv)]{Esposito-Penkov23}. 
Section~\ref{sec:hom-bilin} is dedicated to the study of spaces of linear and bilinear maps and contains mostly new results, among which we would like to emphasize the Adjunction Theorem~\ref{thm:adjunction}. Section~\ref{sec:tensor} contains a systematic study of Beilinson's three topologies on the tensor product. To our knowledge the only other systematic study of these operations available in the literature is due to Positselski~\cite[\S12]{Positselski}, a paper that was immensely helpful to us. Our point of view is different since we place special emphasis on the relationship between tensor products and spaces of linear and bilinear maps. In subsections~\ref{sec:ind-pro}--\ref{sec:tensor-ind-pro} we introduce the key categories of ind-linearly compact, respectively pro-discrete vector spaces, and prove their stability under the completed * and ! tensor products. We owe the idea of working with these categories to Esposito-Penkov, who introduced in~\cite{Esposito-Penkov23} their countable analogues. In Section~\ref{sec:count} we explore the consequences of countability assumptions and prove in particular that the intersection of the categories of countable ind-linearly compact and countable pro-discrete spaces is Tate$^\omega$ (Proposition~\ref{prop:IPTateomega}). This relies on~\cite[Lemma~2.5(i)]{Esposito-Penkov23} and on some new criteria for Tate vector spaces that we establish in subsections~\ref{sec:lim-colim-Tate}--\ref{sec:lim-colim-duality}.


The main theorems on Rabinowitz Floer homology are proved in Sections~\ref{sec:RFH} and~\ref{sec:RFHprodcoprod}: 
Theorem~\ref{thm:main2} corresponds to Theorem~\ref{thm:kappa-top},
the first part of Theorem~\ref{thm:main1} ($RFH_*$ and $RFH^*$ being Tate$^\omega$ and topological duals of each other) to Theorem~\ref{thm:RFH-self-dual}, the first part of
Theorem~\ref{thm:main3} ($\mu$ and $\lambda$ linear continuous defined
on/taking values in the respective completed tensor products) to
Theorem~\ref{thm:RFHmulambdacont}, and the rest of
Theorem~\ref{thm:main1} ($PD$ being a topological isomorphism) and
Theorem~\ref{thm:main3} (intertwining of $\mu$, $\lambda$ with
$\lambda^\vee$, $\mu^\vee$) to Theorem~\ref{thm:RFHmulambdaPD}.
The Appendix discusses some further topics for Tate vector spaces:
splitting of short exact sequences, topological complements,
and the Hahn-Banach theorem.

{\bf Acknowledgements. }
We thank P.~Safronov for bringing Tate vector spaces to our attention,
and D.~Vaintrob, N.~Wahl and V.~Pilloni for inspiring discussions. A.O. thanks Ivan Smith for his hospitality at Cambridge. A.O. is a member of the ANR grant 21-CE40-0002 COSY, and holds a Fellowship of the University of Strasbourg Institute for Advanced Study (USIAS) within the French national programme ``Investment for the future" (IdEx-Unistra).

\section{Linearly topologized vector spaces}\label{sec:lin-top}

We fix a discrete field $\bk$ (i.e., a field equipped with the discrete topology). 
A {\em linearly topologized vector space} is a $\bk$-vector space $V$
with a Hausdorff topology which is translation invariant and has a
basis of open neighbourhoods of $0$ consisting of linear subspaces. 
Linearly topologized vector spaces form a category $\Top$ whose morphisms are
continuous linear maps.
The kernel and cokernel of a morphism $f:V\to W$ are defined as
$$
  \ker(f) = \{v\in V\mid f(v)=0\}, \qquad 
  \coker(f) = W/\ol{f(V)}. 
$$
In particular, quotients in $\Top$ should always be taken by {\em closed}
linear subspaces.

\begin{lemma}\label{lem:open}
An open linear subspace $U\subset V$ of a linearly topologized vector
  space $V$ is also closed. 
Conversely, a closed subspace $L\subset V$ of finite codimension is open. 
\end{lemma}

\begin{proof}
An open linear subspace $U$ is closed because its complement can be
written as a union of translates of $U$. 
Consider now a closed subspace $L\subset V$ of finite codimension.
Pick a basis $v_1,\dots,v_k$ of its complement, so that $V=L\oplus{\rm
  span}\{v_1,\dots,v_k\}$. Since $L$ is closed, we find for 
$j=1,\dots,k$ open linear subspaces $U_j$ such that
$$
  (L\oplus{\rm span}\{v_1,\dots,v_{j-1}\})\cap(v_j+U_j)=\varnothing.
$$
Then $L$ contains the open subspace $U=U_1\cap\cdots\cap U_k$ and is
therefore open. To see this, consider $u\in U$ and write it as
$u=\ell+\sum_{j=1}^k\lambda_jv_j$ with $\ell\in L$ and $\lambda_j\in\bk$. 
If $\lambda_k\neq 0$, then
$v_k-\lambda_k^{-1}u=-\lambda_k^{-1}(\ell+\sum_{j=1}^{k-1}\lambda_jv_j)$ 
would belong to $(L\oplus{\rm span}\{v_1,\dots,v_{k-1}\})\cap(v_k+U_k)$, 
contradicting the displayed equation with $j=k$. Thus $\lambda_k=0$,
and inductively we find $\lambda_j=0$ for all $j$ and therefore $u\in L$.
\end{proof}

\subsection{Limits and colimits}

The category $\Top$ admits products, sums, limits and colimits defined
as follows. 

{\em Products. }The product of a family $\{V_i\}_{i\in I}$ is the \emph{direct product} $\prod_i V_i$ endowed with the initial topology with respect to the canonical projections $proj_i:\prod_i V_i\to V_i$, i.e., the coarsest topology that makes all the maps $proj_i$ continuous. Specifically, the basic open sets in $\prod_i V_i$ are of the form $\prod_{i\notin I_\fin} V_i \times \prod_{i\in I_\fin} \cO_i$, where $I_\fin\subset I$ is any finite subset and $\cO_i\subset V_i$ is open. The direct product satisfies the following universal property: for any $V\in \Top$ and any collection of linear continuous maps $p_i:V\to V_i$, there exists a unique linear continuous map $p:V\to \prod_i V_i$ such that $proj_i\circ p = p_i$. The map $p$ is given by $p(x)=(p_i(x))_{i\in I}$.

{\em Sums. }The sum of a family $\{V_i\}_{i\in I}$ is the \emph{direct sum} $\oplus_i V_i$, endowed with the final topology with respect to the canonical inclusions $incl_i:V_i\to \oplus_i V_i$, i.e., the finest topology that makes all the maps $incl_i$ continuous. Specifically,
the basic open sets in $\oplus_i V_i$ are of the form $\oplus_i\cO_i$, where $\cO_i\subset V_i$ is open.
The direct sum satisfies the following universal property: for any $V\in\Top$ and any collection of linear continuous maps $j_i:V_i\to V$, there exists a unique linear and continuous map $j:\oplus_i V_i\to V$ such that $j\circ incl_i = j_i$. The map $j$ is given by $j(\sum_i x_i) = \sum_i j_i(x_i)$, where the sum is finite.

To define limits and colimits, recall that a {\em quasi ordered set}
$(I,\leq)$ is a set with a reflexive and transitive binary relation.
It is called {\em (upward) directed} if for all $i,j\in I$ there
exists $k\in I$ such $i,j\leq k$.

\emph{Limits}.
Let $(I,\leq)$ be a
directed set
and $\{W_i,g_{ji}\}_{i,j\in I}$ be an inverse system in $\Top$, i.e., we are given linear continuous maps $g_{ji}:W_i\to W_j$ for $i\ge j$ such that $g_{kj}\circ g_{ji}=g_{ki}$ for $i\ge j\ge k$. The \emph{limit} (or \emph{inverse limit}), denoted $\lim W_i$, is a topological vector space endowed with linear continuous maps $g_i:\lim W_i\to W_i$ such that $g_{ji}\circ g_i = g_j$, and that satisfies the following universal property: for any $W\in\Top$ and any collection of maps $\psi_i:W\to W_i$ such that $g_{ji}\circ \psi_i = \psi_j$, there exists a unique linear continuous map $g:W\to \lim W_i$ such that $g_i\circ g=\psi_i$. 
$$
\xymatrix{
W \ar@{.>}[r]^{\exists ! \, g} \ar[dr]_{\{\psi_i\}} & \lim W_i \ar[d]^{\{g_i\}}\\
& \{W_i\}
}
$$
The universal property guarantees that, if it exists, the limit is uniquely determined up to unique isomorphism. One model for the limit is the subspace of $\prod_i W_i$ given by 
$$
\lim W_i = \{ (x_i)\in \prod_i W_i \mid g_{ji}(x_i)=x_j \text{ for all } i\ge j\} \subset \prod_i W_i,
$$
with the family of maps $g_i$ being induced by the canonical projections $\prod W_i \to W_i$, and endowed with the subspace topology, i.e., the coarsest topology that makes the inclusion $\lim W_i\hookrightarrow \prod_i W_i$ continuous. Equivalently, the topology on $\lim W_i$ is the initial topology with respect to the family of maps $g_i:\lim W_i\to W_i$, i.e., the coarsest topology such that all the maps $g_i$ are continuous. The open sets in $\lim W_i$ are intersections with the open sets of $\prod_i W_i$. 

\emph{Colimits}. Let $(I,\leq)$ be a directed set and $\{V_i,f_{ji}\}_{i,j\in I}$ be a direct system in $\Top$, i.e., we are given linear continuous maps $f_{ji}:V_i\to V_j$ for $i\le j$ such that $f_{kj}\circ f_{ji}=f_{ki}$ for all $i\le j\le k$. The \emph{colimit} (or \emph{direct limit}), denoted $\colim V_i$, is a topological vector space endowed with linear continuous maps $f_i:V_i\to V$ such that $f_j\circ f_{ji}=f_i$, and that satisfies the following universal property: for any $V\in \Top$ and any collection of maps $\varphi_i:V_i\to V$ such that $\varphi_j \circ f_{ji}=\varphi_i$, there exists a unique map $f:\colim V_i\to V$ such that $f\circ f_i=\varphi_i$. 
$$
\xymatrix{
\{V_i\} \ar[r]^{\{f_i\}} \ar[dr]_{\{\varphi_i\}} & \colim V_i \ar@{.>}[d]^{\exists ! \, f}\\
& V
}
$$
The universal property guarantees that, if it exists, the colimit is uniquely determined up to unique isomorphism. One model for the colimit is the quotient vector space 
$$
\colim V_i = \oplus_i V_i / \ol{\langle f_{ji} x_i - x_i \mid x_i\in V_i\rangle},
$$
with the family of maps $f_i$ being induced by the canonical inclusions $incl_i:V_i\to \oplus_i V_i$, and 
endowed with the quotient topology, i.e., the finest topology that makes the projection $\oplus_i V_i\to \colim V_i$ continuous. 
The open sets in this topology are by definition the subsets whose preimage under the projection is open. They can be equivalently, and more conveniently, described as the projections of open sets in $\oplus_i V_i$.  
The topology on $\colim V_i$ can also be equivalently described as the final topology with respect to the family of maps $f_i:V_i\to \colim V_i$, i.e., the finest topology such that all the maps $f_i$ are continuous. 

\subsection{Completeness} \label{sec:completeness}

The {\em completion} of a linearly topologized vector space $V$ is
$$
   \wh V := \lim_{U\in\mathcal{U}} V/U
$$
where $\mathcal{U}$ denotes the collection of open linear subspaces in
$V$.
The canonical map $V \to \wh V$ is always injective and its image is dense in $\wh V$.
The space $V$ is called {\em complete} if the canonical map $V \to \wh V$ is an
isomorphism.

\begin{example}
Any discrete vector space is complete because the inverse system
$V/U$, $U\in\mathcal{U}$, has a maximal element equal to $V$.
%
\end{example} 

\begin{lemma}
Let $V$ be complete and $H\subset V$ a linear subspace. Then
$H$ is complete with the induced topology if and only if it is
a closed subspace.
\end{lemma}

\begin{proof}
Let $\cU$ be the collection of open linear subspaces in $V$.
Since $V$ is complete, for any $U\in\cU$ we have $H+U=\lim_{U'\subset
  U,\, U'\in\cU} (H+U)/U'$. This gives the middle isomorphism in the
chain of equalities
$$
  \wh H = \lim_{U\in\cU}H/(H\cap U) = \lim_{U\in\cU}(H+U)/U \cong
  \lim_{U\in\cU}(H+U) = \bigcap_{U\in\cU}(H+U) = \ol H
$$
which proves the lemma.
\end{proof}

Any linear continuous map $f:A\to B$ induces a canonical linear continuous map $\wh f:\wh A\to \wh B$. 

\begin{lemma} \label{lem:completion-dense}
Let $H\subset V$ be a dense linear subspace. Then the canonical map $\wh H\to \wh V$ is an isomorphism. 
\end{lemma}
\begin{proof} Let $\cU$ be the collection of open linear subspaces in $V$.
That $H$ is dense in $V$ is equivalent to $H+U=V$ for any $U\in\cU$. We obtain 
$$
\wh H=\lim_{U\in\cU} H/(H\cap U)= \lim_{U\in\cU}(H+U)/U=\lim_{U\in\cU}V/U=\wh V.
$$
\end{proof}

\begin{example} \label{example:contlinbij}
Let $f:A\to B$ be a continuous linear bijection. The induced map $\wh f:\wh A\to \wh B$, although continuous, may not be a bijection. An explicit example can be constructed as follows. Let $C$ be complete and $B\subset C$ a strict linear subspace which is dense. Let $B_d$ be the space $B$ endowed with the discrete topology and consider the continuous linear bijection $f=\mathrm{Id}:B_d\to B$. Then $\wh f:\wh B_d=B_d\to \wh B=C$ factors as $B_d\stackrel{f}\to B\subset C$ and is a strict inclusion. 
\end{example}

Following Positselski~\cite{Positselski}, let us denote by
$\Top^c\subset\Top$ the full subcategory of complete spaces. 
It is closed under kernels and direct products, hence
under limits, and under direct sums~\cite[Lemma 1.2]{Positselski}. On
the other hand, it is {\em not}
closed under arbitrary cokernels and colimits~\cite[Theorem 2.5]{Positselski}.\footnote{
See however~\S\ref{sec:count} below.}
Therefore, cokernels and coproducts in the
category of complete spaces need to be taken with a completion,
$$
  \wh\coker(f) = \wh{W/\ol{f(V)}},\qquad \wh\colim V_i = \wh{\colim V_i}.
$$

\section{Tate vector spaces}\label{sec:Tate}

In this section we recall some basic notions and facts about \emph{locally linearly compact vector spaces}, or \emph{Tate vector
spaces}. 
Locally linearly compact vector spaces were defined by Lefschetz~\cite[II]{Lefschetz-book} as linear analogues of locally compact groups, in connection with duality theorems. They were used by Tate~\cite{Tate68} in his treatment of residues of differentials on curves, and later revived by Beilinson-Feigin-Mazur~\cite{BFM91} and Beilinson-Drinfeld~\cite{BD04} under the name of Tate vector spaces. A valuable reference for Tate vector spaces is the 2018 dissertation by Rojas~\cite{Rojas}.

There are currently two definitions of these objects in the literature, which we will refer to as the ``Lefschetz definition"~\cite[(II.27.1)]{Lefschetz-book} and the ``Rojas definition"~\cite[Definition~1.19]{Rojas}. Although their equivalence may be known to experts, we were not able to find a written proof of that equivalence. We start by introducing these two definitions, and then prove their equivalence. 

\subsection{Locally linearly compact vector spaces (Lefschetz)} 

Let $V$ be a linearly topologized vector space over the discrete field
$\bk$. A \emph{linear variety in $V$} is a coset mod $H$, for some
linear subspace $H\subset V$. 
Note that a linear variety $S$ that is open is also closed. 

\begin{definition}[Lefschetz~{\cite[(II.27.1)]{Lefschetz-book}}]  \label{defi:Lef-linearly-compact}
A linearly topologized vector space $V$, and more generally a linear variety $S\subset V$, is said to be \emph{linearly compact} whenever given any family $\{S_\alpha\}$ of linear varieties which are closed in $V$, respectively in $S$, such that any finite intersection of the $S_\alpha$ is nonempty, then $\bigcap_\alpha S_\alpha\neq \emptyset$. 
\end{definition}

A linearly compact linear variety is closed~\cite[(II.27.5)]{Lefschetz-book}. 

\begin{definition}[Lefschetz~{\cite[(II.27.9)]{Lefschetz-book}}] 
A linearly topologized vector space $V$ is said to be \emph{locally linearly compact} whenever it has a linearly compact open neighborhood of $0$.  
\end{definition}

\begin{lemma}[Lefschetz~{\cite[(II.27.10)]{Lefschetz-book}}] \label{lem:decomp-Lefschetz}
A vector space $V$ is locally linearly compact if and only if it is a topological direct sum $V= L\oplus D$, with $L$ linearly compact and $D$ discrete. \qed
\end{lemma}

\subsection{Tate vector spaces (Beilinson-Feigin-Mazur, Rojas)}

\begin{definition}[Beilinson-Feigin-Mazur~{\cite[1.2]{BFM91}}, Rojas~{\cite[Definition~1.11]{Rojas}}] 
A linear subspace $L\subset V$ is called \emph{linearly bounded} if $\dim(L/L\cap U)<\infty$ for each open
  linear subspace $U\subset V$.\footnote{Rojas calls this ``linearly compact", but we opted for ``linearly bounded" to distinguish it from Definition~\ref{defi:Lef-linearly-compact}.}
\end{definition}

\begin{example} (i) A discrete vector space is linearly bounded if and only if it is finite dimensional. 

(ii) The vector space $\bk[[t]]$ of formal power series, with a basis of neighborhoods of $0$ given by $t^n\bk[[t]]$, $n\ge 0$, is linearly bounded. 
\end{example}

\begin{definition}[Beilinson-Feigin-Mazur~{\cite[1.2]{BFM91}}, Rojas~{\cite[Definitions~1.11 and~1.15]{Rojas}}] 
A closed linear subspace $L\subset V$ is called 
\begin{itemize}
\item {\em linearly cobounded} if $\dim(V/L+U)<\infty$ for each open
  linear subspace $U\subset V$;
\item a {\em c-lattice} if it is open and linearly bounded;
\item a {\em d-lattice} if it is discrete and linearly cobounded.
\end{itemize}
\end{definition}

The existence of a c-lattice is equivalent to that of a d-lattice: an
algebraic complement of a c-lattice is a d-lattice, and an open
algebraic complement of a d-lattice is a c-lattice.  

\begin{example} Let $V=\bk((t))$ be the vector space of Laurent polynomials, with a basis of neighborhoods of $0$ given by $t^n\bk[[t]]$, $n\in\Z$. Then $\bk[[t]]$ is a c-lattice and $t^{-1}\bk[t^{-1}]$ is a complementary d-lattice. 
\end{example}

\begin{definition}[Beilinson-Feigin-Mazur~{\cite[1.2]{BFM91}}, Rojas~{\cite[Definition~1.19]{Rojas}}] A linearly topologized vector space $V$ is called a {\em Tate vector space} if it is complete and admits a
c-lattice. 
\end{definition}

\begin{example} The space of Laurent polynomials $\bk((t))$ is a Tate vector space. It is neither discrete, nor linearly bounded.
\end{example}

\begin{lemma}\label{lem:top-splitting}
A vector space is Tate iff it is a topological direct sum $V=L\oplus D$ of a
complete c-lattice $L$ and a d-lattice $D$.
\end{lemma}

\begin{proof}
($\Longrightarrow$) Pick a c-lattice $L\subset V$. Then $L$ is open, and thus closed by
translation invariance. It is also complete, being a closed subspace of a complete space. 
Pick any algebraic complement $D$ of $L$. Then $D$ is a d-lattice. 
Moreover, openness of $L$ implies that the projection $V\to L$ along
$D$ is continuous, so its kernel $D$ is closed. 

($\Longleftarrow$) Immediate from the definitions. 
\end{proof}

\subsection{Equivalence of the two definitions}

\begin{proposition} \label{prop:equivalence-Tate-Beilinson}
A linearly topologized vector space is locally linearly compact if and only if it is Tate. 
\end{proposition}

\begin{proof}
In view of the decompositions $L\oplus D$ from Lemmas~\ref{lem:decomp-Lefschetz} and~\ref{lem:top-splitting}, it is enough to prove that a vector subspace $L\subset V$ is linearly compact iff it is closed, complete and linearly bounded. This is the content of Lemma~\ref{lem:equivalence-lin-compact} below. 
\end{proof}
 
\begin{lemma} \label{lem:equivalence-lin-compact}
A vector subspace $L\subset V$ is linearly compact 
iff it is closed, complete and linearly bounded.
\end{lemma}

\begin{proof}
($\Longrightarrow$)
Let $L$ be linearly compact. 
Then $L$ is closed by~\cite[(II.27.5)]{Lefschetz-book}. 

To see that $L$ is complete, consider the canonical linear map $\iota:L\to\wh L$, which is continuous, injective, and whose image is dense in the completion $\wh L$. Now $\iota(L)$ is
linearly compact 
by~\cite[(II.27.4)]{Lefschetz-book},
and therefore closed in $\wh L$ by~\cite[(II.27.5)]{Lefschetz-book}. 
It follows that $\iota:L\to\wh L$ is a continuous linear bijection, and thus
an isomorphism by~\cite[(II.27.8)]{Lefschetz-book}. Hence, $L$ is complete.

For linear boundedness, 
consider an open linear subspace $U\subset L$. Then the space $L/U$ is discrete, it is linearly
compact 
by~\cite[(II.27.4)]{Lefschetz-book}, and therefore finite dimensional by~\cite[(II.27.7)]{Lefschetz-book}. 

($\Longleftarrow$)
Let $L$ be closed, complete and linearly bounded. 
Then $L\simeq \lim_{U\in\cU} L/L\cap U$, and each $L/L\cap U$ is discrete and finite dimensional, hence linearly compact 
by~\cite[(II.27.7)]{Lefschetz-book}. A limit of linearly compact spaces is linearly compact by~\cite[(II.27.6) and (II.27.3)]{Lefschetz-book}, therefore $L$ is linearly compact.
%
\end{proof}

\subsection{Duality}

For a linearly topologized vector space $V$ we denote by $V^*$ its topological
dual. Following Lefschetz~\cite[(II.28.1)]{Lefschetz-book}, we equip $V^*$ with the linear topology whose neighbourhood
basis of the origin is given by the linear subspaces
$$
   L^\perp=\{\alpha\in V^*\mid \alpha|_L=0\},\qquad
   L\subset V \text{ linearly compact subspace}.
$$
This is precisely the compact-open topology on $\Hom(V,\bk)$, see~\S\ref{ss:hom}.

A continuous linear map $f:V\to W$ induces a continuous linear map $f^*:W^*\to V^*$, $f^*\beta=\beta\circ f$. To prove continuity of $f^*$, consider an open subspace $L^\perp\subset V^*$, with $L\subset V$ linearly compact. Then $(f^*)^{-1}(L^\perp)=f(L)^\perp$, and this is open because $f(L)$ is linearly compact~\cite[(II.27.8)]{Lefschetz-book}.

\begin{theorem}[Lefschetz-Tate duality, Lefschetz~{\cite[(II.28.2-29.1)]{Lefschetz-book}}, Rojas~{\cite[Theorem~1.25]{Rojas}}]\label{thm:Tate-duality}\quad 

(a) If $V$ is discrete, then $V^*$ is linearly compact. 

(b) If $V$ is linearly compact, then $V^*$ is discrete.
  
(c) If $V$ is locally linearly compact, then so is $V^*$ and the canonical map $V\to
  V^{**}$ is a topological isomorphism. \qed
\end{theorem}
   
\begin{corollary} \label{cor:duality}
Let $f:V\to W$ be continuous linear and $V$, $W$ locally linearly compact. Then $f$ is an isomorphism if and only if $f^*:W^*\to V^*$ is an isomorphism. 
\end{corollary}   

\begin{proof}
($\Longrightarrow$)
This implication does not use that $V$, $W$ are locally linearly compact, and follows from the fact that taking the dual is a functorial operation. Let $g:W\to V$ be the inverse of $f$. Then $g^*$ is an inverse for $f^*$, which shows that $f^*$ is an isomorphism. 

($\Longleftarrow$)
If $f^*$ is an isomorphism, then $f^{**}:V^{**}\to W^{**}$ is an isomorphism by the direct implication. Under the assumption that $V$, $W$ are locally linearly compact, we have $V\simeq V^{**}$, $W\simeq W^{**}$ and the map $f^{**}$ is identified with $f$, therefore $f$ is an isomorphism.  
\end{proof}

The following result is an abstract version of the splitting theorem
for Rabinowitz Floer homology from~\cite{CHO-reducedSH}. It was proved for Tate$^\omega$ vector spaces by Esposito-Penkov in~\cite[Lemma 2.5(iv)]{Esposito-Penkov23}.

\begin{proposition}[Self-duality]\label{prop:self-duality}
An infinite dimensional Tate vector space $V$ satisfies $V\cong V^*$
if and only if $V\cong D\oplus D^*$ for an infinite dimensional
discrete vector space $D$.
\end{proposition}

\begin{proof}
If $V\cong D\oplus D^*$ then it is clearly self-dual. Conversely, assume
that there exists an isomorphism $\phi:V\to V^*$. Write $V=L\oplus E$
with $L$ open linearly compact and $E$ discrete. Then
$L^\perp\subset V^*$ is open, and it is also linearly compact because
$L^\perp=\{f\in V^*\mid f|_L=0\}\cong E^*$. Therefore,
$\phi^{-1}(L^\perp)$ and $K:=L\cap\phi^{-1}(L^\perp)$ are open
linearly compact and we can write $V=K\oplus D$ with $D$ discrete.
Note that $D$ is infinite dimensional (otherwise $K\oplus D=V\cong V^*\cong
K^*\oplus D^*$ would be linearly compact and discrete, contradicting
infinite dimensionality of $V$). 

From $\phi(K) = \phi(L)\cap L^\perp \subset L^\perp\subset K^\perp$ we
infer $K\subset\phi^{-1}(K^\perp)$. 
Since $K$ and $\phi^{-1}(K^\perp)$ are both open linearly compact, we
have $\phi^{-1}(K^\perp)=K\oplus F$ with $F$ finite dimensional. Thus
$K\oplus F\cong_{\phi} K^\perp\cong D^*$, and dualizing we obtain
$D\cong K^*\oplus F^*\cong K^*$, where the last isomorphism holds
because $D$ is discrete of infinite dimension and $F^*$ is finite
dimensional. Hence, $V=K\oplus D\cong D^*\oplus D$. 
\end{proof}

\begin{remark}  
Many of the results in this section have analogues in the
Pontrjagin--van Kampen duality theory for locally compact abelian
groups, see~\cite[Chapter II]{Lefschetz-book}.
\end{remark}

\section{Homomorphisms and bilinear maps}\label{sec:hom-bilin}

Here we discuss various notions of homomorphisms and
their relations. Throughout this section $A,B,C$ denote linearly
topologized vector spaces over a discrete field $\bk$. 

\subsection{Homomorphisms}\label{ss:hom}

We denote by $\Hom(A,B)$ the space of continuous linear maps $f:A\to B$.
This becomes a linearly topologized vector space with the {\em 
  compact-open topology} whose neighbourhood basis of zero consists of
the sets 
$$
  S_{K,V}:=\{f\in\Hom(A,B)\mid f(K)\subset V\}
$$
for $K\subset A$ linearly compact and $V\subset B$ linear open. Note
that for $B=\bk$ this recovers the topology on $A^*$ defined in~\S\ref{sec:Tate}.  

\begin{lemma}\label{lem:Hom1}
(a) If $A$ is linearly compact and $B$ is discrete, then $\Hom(A,B)$ is
  discrete. \\
(b) If $A$ is discrete and $B$ is linearly compact, then $\Hom(A,B)$ is
  linearly compact. \\
(c) If $A$ and $B$ are discrete and infinite dimensional, then
  $\Hom(A,B)$ is {\em not} Tate. \\
(d) If $A$ and $B$ are linearly compact such that each nonzero open
  linear subspace contains a strictly smaller open subspace, then
  $\Hom(A,B)$ is {\em not} Tate. \\
(e) If $A$ is Tate, then the subset $\Hom_{\rm fin}(A,B)$ of finite rank operators is dense in $\Hom(A,B)$. 
\end{lemma} 

\begin{proof}
Part (a) holds because $S_{A,\{0\}}=\{0\}$ is open.
For (b), consider $K\subset A$ linearly compact and $V\subset B$
linear open with projection $\pi_V:B\to B/V$. Then $S_{K,V}$ is the
kernel of the surjective map\footnote{The map $\Hom(A,B)\to \Hom(K,B)$ is surjective because $A$ is discrete, and the map $\Hom(K,B)\to \Hom(K,B/V)$ is surjective because $K$ is discrete.} 
$$
  \Hom(A,B)\to\Hom(K,B/V),\quad f\mapsto \pi_V\circ f|_K,
$$
and therefore $\Hom(A,B)/S_{K,V}\cong \Hom(K,B/V)$. This is finite
dimensional because under the hypotheses both $K$ and $B/V$ are. This proves that $\Hom(A,B)$ is linearly bounded. That $\Hom(A,B)$ is also complete follows from Lemma~\ref{lem:Hom2}(b): the discrete space $A$ is linearly compactly generated since it is the colimit of its finite dimensional subspaces, and the linearly compact space $B$ is complete.

Alternative proof of (b): We write $A=\oplus_I \bk$, where $I$ is an indexing set for a basis of $A$. Then 
$$
\Hom(A,B)=\Hom(\oplus_I \bk,B)=\prod_I\Hom(\bk,B)=\prod_I B,
$$
and we conclude using the fact that a product of linearly compact spaces is linearly compact. The second equality above is Lemma~\ref{lem:hom-lim}(b) below.\footnote{Note that $\Hom(\oplus A_i,B)=\prod \Hom(A_i,B)$ for any $(A_i)$ and $B$, without the Tate assumption from Lemma~\ref{lem:hom-lim}(b).}

For (c), note that the basic open subspaces are $K^\perp=S_{K,\{0\}}$ for
$K\subset A$ linearly compact, i.e., finite dimensional. For another
linearly compact subspace $L\subset A$, the short exact sequence $0\to
(K+L)/K\to A/K\to A/(K+L)\to 0$ gives rise to a commuting diagram
$$
\xymatrix
@C=15pt
{
0 \ar[r] & (K+L)^\perp \ar[d]^\cong \ar[r] & K^\perp \ar[d]^\cong \ar[r] & K^\perp/(K+L)^\perp
\ar[d]^\cong \ar[r] & 0 \\
0 \ar[r] & \Hom(\frac{A}{K+L},B) \ar[r] & \Hom(\frac{A}{K},B) \ar[r] & \Hom(\frac{K+L}{K},B)
\ar[r] & 0 \\
}
$$
Since $\dim B=\infty$, the quotient space $K^\perp/(K^\perp\cap L^\perp)\cong
\Hom(\frac{K+L}{K},B)$ is infinite dimensional unless $L\subset K$.
This means that $K^\perp$ can only be linearly compact if $K$ contains
all finite dimensional subspaces of $A$, which is impossible because
$\dim A=\infty$. 
The proof of (d) is similar.

For (e), we need to show that for each $f\in\Hom(A,B)$, $K\subset A$
linearly compact, and $V\subset B$ linear open, the set $f-S_{K,V}$
contains a finite rank operator. Since $A$ is Tate, each linearly
compact set $K$ is contained in an open linearly compact set $K'$
(take $K'=K+U$ for an open linearly compact set $U$). After replacing
$K$ by $K'$, we may thus assume that $K$ is open. 
Now $U:=f^{-1}(V)$ is open and  
$K/(K\cap U)$ is finite dimensional. Pick a complement $D\subset K$ of
$K\cap U$ and define $g\in\Hom_f(K,B)$ by $g|_D=f|_D$ and $g|_{K\cap U}=0$.
Since $K$ is open, it admits a topological complement $E$ and we can
extend $g$ to a finite rank operator $g\in\Hom_f(A,B)$ by setting it
$0$ on $E$. Then $(f-g)(K)\subset V$, so the open set
$f-S_{K,V}$ contains the finite rank operator $g$. 
\end{proof}

\begin{remark}
Parts (c) and (d) of Lemma~\ref{lem:Hom1} show that
$\Hom(A,B)$ need not be Tate if $A$ and $B$ are Tate. 
\end{remark}

The next lemma gives some conditions under which the $\Hom$ functor
commutes with limits or colimits. 

\begin{lemma}\label{lem:hom-lim}
(a) For each inverse system $(B_i)$ and each $A$ there is a canonical
  topological isomorphism  
$$
  \Hom(A,\lim B_i) \stackrel{\simeq}{\longrightarrow} \lim\Hom(A,B_i).
$$
(b) For each direct system of {\em Tate} vector spaces $(A_i)$ and each $B$
  there is a canonical topological isomorphism  
$$
  \Hom(\colim A_i,B) \stackrel{\simeq}{\longrightarrow} \lim\Hom(A_i,B).
$$
(c) For each inverse system $(A_i)$ and each {\em discrete} $B$
  there is a canonical topological isomorphism  
$$
  \colim\Hom(A_i,B) \stackrel{\simeq}{\longrightarrow} \Hom(\lim A_i,B).
$$
\end{lemma}

\begin{proof}
(a) With $B=\lim B_i$ the inverse system is given for $i\leq j$ by the
  commuting diagram 
$$
\xymatrix{
  B_i & \ar[l]_{g_{ij}} B_j \\ B \ar[u]^{g_i} \ar[ru]_{g_j} 
}
$$
It induces the commuting diagram
$$
\xymatrix{
  \Hom(A,B_i) & \ar[l]_{\wh g_{ij}} \Hom(A,B_j) \\ \Hom(A,B)
  \ar[u]^{\wh g_i} \ar[ru]_{\wh g_j} 
}
$$
with $\wh g_{ij}(f_j)=g_{ij}\circ f_j$ and $\wh g_i(f)=g_i\circ f$,
and thus a canonical map
$$
\wh g:\Hom(A,B)\to \lim\Hom(A,B_i). 
$$
It is straightforward to check that $\wh g$ is a linear bijection, so
it only remains to compare the topologies. The basic linear open
subspaces in $\Hom(A,B)$ are $S_{K,V}$ with $K\subset A$ linearly
compact and $V\subset B$ linear open. Thus, $V=\prod_{j\in J}V_j\cap B$ (the
components for $i\notin J$ are $B_i$ and suppressed from the notation)
with $V_j\subset B_j$ linear open and $J\subset I$ finite.
The basic linear open subspaces of $\lim\Hom(A,B_i)$ are $\prod_{j\in
  J}S_{K_j,V_j}$ with $K_j\subset A$ linearly compact and $V_j\subset
B_j$ linear open, $J\subset I$ finite. Since
$S_{K,V}\supset\prod_jS_{K,V_j}$ and $\prod_jS_{K_j,V_j}\supset
S_{\sum K_j,V}$,\footnote{If $(K_j)$ is a finite family of linearly compact subspaces, then $\sum K_j$ is linearly compact.} 
these correspond to the same topology under $\wh g$.

(b) With $A=\colim A_i$ the direct system is given for $i\leq j$ by the
  commuting diagram 
$$
\xymatrix{
  A_i \ar[d]_{f_i} \ar[r]^{f_{ji}} & A_j \ar[dl]^{f_j} \\ A  
}
$$
It induces the commuting diagram
$$
\xymatrix{
  \Hom(A_i,B) & \ar[l]_{\wh f_{ji}} \Hom(A_j,B) \\ \Hom(A,B)
  \ar[u]^{\wh f_i} \ar[ru]_{\wh f_j} 
}
$$
with $\wh f_{ji}(g_j)=g_j\circ f_{ji}$ and $\wh f_i(g)=g\circ f_i$,
and thus a canonical map
$$
\wh f:\Hom(A,B)\to \lim\Hom(A_i,B). 
$$
By definition $\wh f$ is continuous. It is straightforward to check that $\wh f$ is a linear bijection, so
it only remains to compare the topologies.
The basic linear open subspaces of $\lim\Hom(A_i,B)$ are $\prod_{j\in
  J}S_{K_j,V_j} \cap \lim\Hom(A_i,B)$ with $K_j\subset A_j$ linearly compact and $V_j\subset
B$ linear open, $J\subset I$ finite. After mapping the $K_j$ to some
$B_k$ with $k\geq j$ for all $j\in J$, we may assume that $J$ consists
of a single element $j$.   

The basic linear open subspaces in $\Hom(A,B)$ are $S_{K,V}$ with
$K\subset A$ linearly compact and $V\subset B$ linear open. 
By Lemma~\ref{lem:compact-in-colim} below, there exists some $j$
such that $K\subset f_j(A_j)$.
Consider an open subset $U=\oplus U_i$ in $A$.\footnote{Strictly speaking, an open set is $U=\mathrm{pr}(\oplus U_i)$, where $\mathrm{pr}:\oplus A_i\to \colim A_i$ is the canonical projection, but we simply denote this $\oplus U_i$.} The projection
$\pi:A\to A/U$ maps $K$ onto the finite dimensional space $\pi(K)=K/(K\cap U)$. 
Pick a finite dimensional subspace $D\subset K$ with $\pi(D)=\pi(K)$,
and a finite dimensional subspace $D_j\subset A_j$ with $f_j(D_j)=D$.
Then for each $a\in K$ there exists $d\in D$ with $\pi(d)=\pi(a)$,
thus $d-a\in U$. Writing $a=f_j(a_j)$ and $d=f_j(d_j)$ with $a_j\in
A_j$ and $d_j\in D_j$, we conclude $f_j(a_j-d_j)\in U$ and therefore
$a_j-d_j\in U_j$, thus $a_j\in U_j+D_j$ and $f_j(a_j)=a$. This shows
that $K\subset f_j(U_j+D_j)$. 
Now we invoke the hypothesis that $A_j$ is Tate. It allows us to
take $U_j$ to be open and linearly compact, so that $U_j+D_j$ is
linearly compact. Then $K=f_j(K_j)$ for the linearly compact subset
$K_j=(U_j+D_j)\cap f_j^{-1}(K)$, and we see that the topologies match.  
%
%
%

(c) With $A=\lim A_i$ the inverse system is given for $i\leq j$ by the
  commuting diagram 
$$
\xymatrix{
  A_i & \ar[l]_{f_{ij}} A_j \\ A \ar[u]^{f_i} \ar[ru]_{f_j} 
}
$$
It induces the commuting diagram
$$
\xymatrix{
  \Hom(A_i,B) \ar[d]_{\wh f_i} \ar[r]^{\wh f_{ij}} & \Hom(A_j,B)
  \ar[ld]^{\wh f_j} \\ \Hom(A,B) 
}
$$
with $\wh f_{ij}(g_i)=g_i\circ f_{ij}$ and $\wh f_i(g_i)=g_i\circ f_i$,
and thus a canonical map
$$
\wh f:\colim\Hom(A_i,B)\to \Hom(A,B). 
$$
This map is continuous since it is obtained using the universal properties of $\lim$ and $\colim$. It is straightforward to check that $\wh f$ is a linear bijection, so
it only remains to compare the topologies.
For this, we use the hypothesis that $B$ is discrete. It implies that 
the basic linear open subspaces of $\colim\Hom(A_i,B)$ are $\oplus
K_i^\perp$ with $K_i\subset A_i$ linearly compact, where
$K_i^\perp=S_{K_i,\{0\}}$.
The basic linear open subspaces in $\Hom(A,B)$ are
$K^\perp=S_{K,\{0\}}$ with $K\subset A$ linearly compact. 
This means that $K\subset \prod K_i\cap A$ with $K_i\subset A_i$ linearly compact,
hence $K^\perp \supset \oplus K_i^\perp$ and we see that the topologies match.  
\end{proof}

In the preceding proof we have used the following lemma.

\begin{lemma}\label{lem:compact-in-colim}
Consider a direct system $(A_i,f_{ji})$ with the canonical maps $f_j:A_j\to
\colim A_i$. Then for each linearly compact subspace
$K\subset\colim A_i$ there exists some $j$ such that $K\subset f_j(A_j)$.
\end{lemma}

\begin{proof}
Assume this is not the case and construct a sequence $(x_k)$ in $K$ as
follows. Take any $x_1\in K\setminus\{0\}$, and represent it by
$a_1\in A_{i_1}$, i.e., $x_1=f_{i_1}(a_1)$.  Take any $x_2\in
K\setminus f_{i_1}(A_{i_1})$, which is nonempty by assumption, and
represent it by $a_2\in A_{i_2}$ for $i_2>i_1$. By induction, having
constructed $x_1,\dots,x_k$ and $i_1<\dots<i_k$ such that $x_s\in
K\cap f_{i_s}(A_{i_s})\setminus f_{i_{s-1}}(A_{i_{s-1}})$ for all $s$,
we take any $x_{k+1}\in K\setminus f_{i_k}(A_{i_k})$ and represent it
by $a_{k+1}\in A_{i_{k+1}}$ for some $i_{k+1}>i_k$. Then for each $k$
we have $x_k=f_{i_k}(a_k)\in K\cap f_{i_k}(A_{i_k})\setminus
f_{i_{k-1}}(A_{i_{k-1}})$ with $a_k\in A_{i_k}$. Choose open linear
subspaces $U_{i_k}\subset A_{i_k}$ such that $a_k\notin U_{i_k}$. Let
$\pi:\oplus A_i\to \colim A_i$ be the canonical projection, and
consider the open subspace  
$$
   U := \pi\Bigl(\prod_{k}U_{i_k}\times
   \prod_{i\neq i_k \, \forall k}A_i\Bigr)\subset \colim A_i.
$$
By construction, the projections $[x_1],[x_2],\dots\in K/(K\cap U)$ are
linearly independent. But then $\dim(K/(K\cap U))=\infty$, contradicting
linear compactness of $K$.
\end{proof}

\begin{remark}
Parts (a) and (b) in Lemma~\ref{lem:hom-lim} are expected from general
properties of limits/colimits and $\Hom$ functors, see
e.g.~\cite[Propositions 5.21 and 5.26]{Rotman}. However, the matching
of the topologies does not seem to follow from abstract category
theory, and in (b) appears to require the additional Tate hypothesis (or, alternatively, using Corollary~\ref{cor:exists-closed-complement} in the Appendix, the assumption that the spaces involved in the direct system are complete and have countable basis).
\end{remark}

We will repeatedly use the following special case of
Lemma~\ref{lem:hom-lim}(b) and (c): 

\begin{corollary}\label{cor:lim-colim-duality}
(a) For each direct system of {\em Tate} vector spaces $A_i$ there is a
  canonical topological isomorphism
$$
  (\colim A_i)^* \stackrel{\simeq}{\longrightarrow} \lim A_i^*.
$$
(b) For each inverse system of linearly topologized vector spaces $A_i$ there is a
  canonical topological isomorphism
$$
  \colim A_i^* \stackrel{\simeq}{\longrightarrow} (\lim A_i)^*.
$$
\qed
\end{corollary}

Let us call $A$ {\em linearly compactly generated} if the canonical map
$$
  \iota:\colim_{K\subset A\text{ linearly compact}}K\to A
$$
is a topological isomorphism. Note that $\iota$ is always continuous
and bijective, so this condition is equivalent to saying that
$U\subset A$ is open if $U\cap K$ is open for all $K\subset A$ compact.\footnote{
The notions ``compact-open topology'' and  ``linearly compactly
generated'' are analogues of the corresponding notions in topology
and our discussion largely parallels the one
in~\cite[Appendix]{Hatcher}. Similarly for the notion of ``locally
linearly compact".} 
The following lemma shows that compact generation is in some sense
dual to completeness. 

\begin{lemma}\label{lem:Hom2}
(a) Each Tate vector space $A$ is complete and linearly compactly generated. \\
(b) If the space $A$ is linearly compactly generated and $B$ is complete, then
  $\Hom(A,B)$ is complete. In particular, $A^*$ is complete. \\
(c) If $A$ is complete, then $A^*$ is linearly compactly generated.
\end{lemma}

\begin{proof}
(a) Completeness holds by definition. 
To show that $A$ is linearly compactly generated, consider $U\subset A$
such that $U\cap K$ is open for all $K\subset A$ 
compact. Take a c-lattice $L\subset A$. Since $L$ is compact and open,
we have $U\cap L\subset_{\rm open} L\subset_{\rm open} A$. Thus $U\cap
L$ is open in $A$ and $U\cap L\subset U$, hence $U$ is open.  

(b) For $K\subset A$ linearly compact and $V\subset B$ linear open, in
the proof of Lemma~\ref{lem:Hom1} we have seen $\Hom(A,B)/S_{K,V}\cong
\Hom(K,B/V)$. Therefore, we get topological isomorphisms 
\begin{align*}
  \lim_{K,V}\Hom(A,B)/S_{K,V}
  &\cong \lim_V\lim_K\Hom(K,B/V) \cr
  &\cong \lim_V\Hom(\colim_KK,B/V) \cr
  &\cong \lim_V\Hom(A,B/V) \cr
  &\cong \Hom(A,\lim_V B/V) \cr
  &\cong \Hom(A,B).
\end{align*}
Here the first isomorphism holds because limits commute
(see~\cite[Theorem 5.1]{Hilton-Stammbach}); the second one by
Lemma~\ref{lem:hom-lim}(b), using that $K$ is linearly compact and
therefore Tate; the third one because $A$ is linearly compactly generated;
the fourth one by Lemma~\ref{lem:hom-lim}(a); and the fifth one
because $B$ is complete.

(c) Consider first general spaces $A,B$. 
For $U\subset A$ open and $L\subset B$ linearly compact, consider the
linear bijection
\begin{equation}\label{eq:TUL}
  \Hom(A/U,L)\to T_{U,L} := \{f\in\Hom(A,B)\mid f|_U=0\text{ and }f(A)\subset L\}
\end{equation}
sending $f$ to $\iota\circ f\circ\pi$, where $\pi:A\to A/U$ is the
projection and $\iota:L\to B$ the inclusion. Now linear open
subsets $\ol V\subset L$ are of the form $\ol V=\iota^{-1}(V)$ for
$V\subset B$ linear open, and linearly compact (hence finite
dimensional) subsets $\ol K\subset A/U$ are of the form $\pi(K)$ for
$K\subset A$ linearly compact. Since the above linear bijection sends
$S_{\ol K,\ol V}$ onto $S_{K,V}$, it is an isomorphism. 
Since $A/U$ is discrete and $L$ is linearly compact,
$\Hom(A/U,L)$ is linearly compact by Lemma~\ref{lem:Hom1}(b). 
Now assume that $A$ is complete and $B=\bk$. Then $A=\lim_{U\text{
    linear open}}A/U$ and we get topological isomorphisms 
\begin{align*}
  \colim_U\Hom(A/U,\bk) 
  \cong \Hom(\lim_U A/U,\bk)
  \cong \Hom(A,\bk) = A^*,
\end{align*}
where the first isomorphism follows from Lemma~\ref{lem:hom-lim}(c) with $B=\bk$.
Since $A^*\supset T_{U,\bk}\cong\Hom(A/U,\bk)$ is linearly compact,
this shows that $A^*$ is linearly compactly generated. 
\end{proof}

\subsection{Compact, discrete, and finite rank homomorphisms}\label{ss:comp-disc-fin}

This subsection follows Beilinson-Drinfeld~\cite[\S 2.7.7]{BD04}.
The space $\Hom(A,B)$ contains several natural subspaces: the space $\Hom_c(A,B)$ of compact homomorphisms (whose
image has linearly compact closure), the space $\Hom_d(A,B)$ of discrete
homomorphisms (whose kernel is open), the space $\Hom_f(A,B)=\Hom_c(A,B)\cap\Hom_d(A,B)$, and also the space
$\Hom_\fin(A,B)\subset \Hom_f(A,B)$ of finite rank homomorphisms (whose image is finite
dimensional). We equip these spaces with the linear
topologies generated respectively by the linear subspaces
$\Hom_c(A,V)$, $\Hom_d(A/K,B)$, and $\Hom_f(A/K,V)$ for $K\subset A$
linearly compact and $V\subset B$ linearly open,
where $\Hom_d(A/K,B)$ and $\Hom_f(A/K,V)$ are viewed as subspaces of
$\Hom(A,B)$ via the isomorphism~\eqref{eq:TUL}. With these
topologies, the spaces fit into the short exact sequence
\begin{equation}\label{eq:hom-ses}
\xymatrix
@C=12pt
{
  0 \ar[r] & \Hom_f(A,B) \ar[r]^{\hspace{-40pt}\alpha} & \Hom_c(A,B)\oplus\Hom_d(A,B)
  \ar[r]^{\hspace{40pt}\beta} & \Hom(A,B)
}
\end{equation}
where $\alpha(f)=(f,f)$ is a homeomorphism onto a closed subspace and
$\beta(g,h)=g-h$ is continuous. If $B$ is Tate, then $\beta$ is open
and surjective (see~\cite[Proposition 1.33]{Rojas}). 

\begin{example}
Let $A$ be an infinite dimensional linearly compact space and denote
by $A_d$ the same vector space equipped with the discrete topology.
Then the identity map $A\to A$ belongs to $\Hom_f(A_d,A)$ but not to
$\Hom_\fin(A_d,A)$, so the inclusion $\Hom_\fin(A_d,A)\subset\Hom_f(A_d,A)$ is strict.
\end{example}

We will need the following variation of Lemma~\ref{lem:hom-lim}.

\begin{lemma}\label{lem:hom-lim2}
(a) For each inverse system $(B_i)$ and each $A$ there is a canonical
  topological isomorphism  
$$
  \Hom_c(A,\lim B_i) \cong \lim\Hom_c(A,B_i).
$$
(b) For each direct system of {\em Tate} vector spaces $(A_i)$ and each $B$
  there is a canonical topological isomorphism  
$$
  \Hom_d(\colim A_i,B) \cong \lim\Hom_d(A_i,B).
$$
\end{lemma}

\begin{proof}
(a) Set $B=\lim B_i$ and recall from the proof of
  Lemma~\ref{lem:hom-lim}(a) the isomorphism 
$$
  \wh g:\Hom(A,B)\stackrel{\cong}\longrightarrow \lim\Hom(A,B_i)
$$
sending $f$ to $(f_i)$ with $f_i=g_i\circ f$. If $f$ is
linearly compact, then so are the $f_i$.\footnote{We have $\overline{\im(g_i\circ f)}=\overline{g_i(f(A))}=\overline{g_i(\overline{f(A)})}=g_i(\overline{f(A)})$, which is linearly compact if $\overline{f(A)}$ is linearly compact.}
Conversely, if the $f_i$ are
compact, then $\ol{\im f} = \lim\ol{\im f_i}$ is linearly compact by
Lemma~\ref{lem:colim-lim-Tate}(b), hence $f$ is compact. Thus $\wh g$
induces a linear bijection
$$
  \wh g:\Hom_c(A,B)\stackrel{\cong}\longrightarrow \lim\Hom_c(A,B_i).
$$
The topologies match because $\lim\Hom_c(A,B_i)$ has the basic open
sets $\prod_{j\in J}\Hom_c(A,V_i)$ with $V_i\subset B_i$ open linear
and $J$ finite, whereas $\Hom_c(A,B)$ has the basic open sets
$\Hom_c(A,V)$ with $V\subset B$ open linear, i.e., $V=\prod_{j\in
  J}V_i$ with $V_i\subset B_i$ open linear and $J$ finite.

(b) Set $A=\colim A_i$ and recall from the proof of
  Lemma~\ref{lem:hom-lim}(b) (here we use that the $A_i$ are Tate) the isomorphism 
$$
  \wh f:\Hom(A,B)\stackrel{\cong}\longrightarrow \lim\Hom(A_i,B). 
$$
sending $g$ to $(g_i)$ with $g_i=g\circ f_i$. If $\ker g$ is open,
then so are the $\ker g_i=f_i^{-1}(\ker g)$. Conversely, suppose that
the $\ker g_i$ are open and consider the short exact sequence of
direct systems
$$
  0 \to \ker g_i \to A_i \to A_i/\ker g_i \to 0.
$$
Since the colimit is an exact functor~\cite[Theorem 2.6.15]{Weibel},
we get an exact sequence
$$
  0 \to \ker g=\colim\ker g_i \to A=\colim A_i \stackrel{p}\to
  \colim(A_i/\ker g_i) \to 0. 
$$
Since $A_i/\ker g_i$ are discrete, so is $\colim(A_i/\ker g_i)$ by
Lemma~\ref{lem:colim-lim-Tate}(a), hence $\ker g=\ker p$ is open.  
Thus $\wh f$ induces a linear bijection
$$
  \wh f:\Hom_d(A,B)\stackrel{\cong}\longrightarrow \lim\Hom_d(A,B_i). 
$$
To see that the topologies match, note that $\lim\Hom_d(A_i,B)$ has the basic open
sets $\prod_{j\in J}\Hom_d(A_i/K_i,B)$ with $K_i\subset A_i$ linearly compact
and $J$ finite (w.l.o.g.~$J=\{j\}$), while $\Hom_d(A,B)$ has the basic open sets
$\Hom_d(A/K,B)$ with $K\subset A$ linearly compact. By the proof of
Lemma~\ref{lem:hom-lim}(b) we have $K=f_j(K_j)$ for a linearly compact subspace
$K_j\subset A_j$, so the topologies match.
\end{proof}

\begin{lemma}
(a) If $B$ is complete, then $\Hom_c(A,B)$ is complete. \\
(b) If $A$ is Tate, then $\Hom_d(A,B)$ is complete. \\
(c) If $A$ is Tate and $B$ is complete, then $\Hom_f(A,B)$ is complete.
\end{lemma}

\begin{proof}
(a) For $V\subset B$ linear open we have a short exact sequence $0\to
  \Hom_c(A,V)\to \Hom_c(A,B)\to \Hom_c(A,B/V)\to 0$. Thus 
\begin{align*}
  \lim_V\frac{\Hom_c(A,B)}{\Hom_c(A,V)}
  &\cong \lim_V\Hom_c(A,B/V) \cr
  &\cong \Hom_c(A,\lim_V B/V) \cong \Hom_c(A,B),
\end{align*}
where the second isomorphism follows from Lemma~\ref{lem:hom-lim2}(a)
and the third one from completeness of $B$. 

(b) For $K\subset A$ linearly compact we have a short exact sequence $0\to
  \Hom_d(A/K,B)\to \Hom_d(A,B)\to \Hom_d(K,B)\to 0$. Thus 
\begin{align*}
  \lim_K\frac{\Hom_d(A,B)}{\Hom_d(A/K,B)}
  &\cong \lim_K\Hom_d(K,B) \cr
  &\cong \Hom_d(\colim_KK,B) \cong \Hom_d(A,B),
\end{align*}
where the second isomorphism follows from Lemma~\ref{lem:hom-lim2}(b)
(using that $K$ is linearly compact and thus Tate)
and the third one from compact generation of $A$. 

(c) Completeness of $\Hom_f(A,B)$ holds because by~\eqref{eq:hom-ses}
it is a closed subspace of $\Hom_c(A,B)\oplus\Hom_d(A,B)$, which is
complete by parts (a) and (b). 
\end{proof}

\subsection{Bilinear maps}\label{ss:bilin}

The following discussion of bilinear maps will be related to tensor products in~\S\ref{sec:tensor}.

\begin{lemma}\label{lem:bil-cont}
A bilinear map $\phi:A\times B\to C$ is continuous (with respect to the product topology on $A\times B$) if and only if, for any $W\subset C$ linear open, the following conditions are satisfied: 
\begin{enumerate}[label=(\roman*)]
\item there exist open linear subspaces $U\subset A$ and $V\subset B$
  such that $U\times V\subset \phi^{-1}(W)$;  
\item for each $a\in A$ there exists an open linear subspace $V\subset
  B$ such that $a\times V\subset \phi^{-1}(W)$;
\item for each $b\in B$ there exists an open linear subspace $U\subset
  A$ such that $U\times b\subset \phi^{-1}(W)$.
\end{enumerate} 
\end{lemma}

In particular, continuity of $\phi$ is a stronger condition than just requiring $\phi^{-1}(W)$ to be open for any linear open subspace $W\subset C$ (which would amount to the first among the three conditions above). This difference between the continuity condition for bilinear maps and the one for linear maps arises because the fibers of a bilinear map are generally not affine subspaces of $A\times B$. 

\begin{proof}[Proof of Lemma~\ref{lem:bil-cont}]
Continuity of $\phi$ is equivalent to $\phi^{-1}(c+W)$ being open for
any $c\in C$ and $W\subset C$ linear open. This means that for every
$(a,b)\in A\times B$ and $W\subset C$ linear open there exist
$U\subset A$, $V\subset B$ linear open such that
$(a+U)\times(b+V)\subset \phi^{-1}(c+W)$ with $c=\phi(a,b)$, i.e., 
for any $u\in U$ and $v\in V$ we have
$\phi(a,v)+\phi(u,b)+\phi(u,v)\in W$.
Taking $a=b=0$, $u=0$ and $v=0$, respectively, yields conditions (i-iii).
The converse is immediate.
\end{proof}

Continuity of a bilinear map is a strong constraint. The following result gives an example.

\begin{proposition}
The canonical evaluation map 
$$
\ev:A\times A^*\to \bk,\qquad (a,f)\mapsto f(a)
$$ 
is continuous as a bilinear map if and only if $A$ is Tate. 
\end{proposition}

\begin{proof}
A basis of open neighborhoods of $0$ for $A^*$ is given by $K^\perp$ with $K\subset A$ linearly compact. By Lemma~\ref{lem:bil-cont} continuity of $\ev$ is equivalent to the existence of $U$ open and $K$ linearly compact such that $U\times K^\perp\subset\ev^{-1}(0)$, and the existence, for any $a\in A$, $f\in A^*$, of $U\subset A$ linear open and $K\subset A$ linearly compact such that $a\times K^\perp, U \times f \subset \ev^{-1}(0)$. 

Assume $\ev$ continuous. The condition $U\times K^\perp\subset
\ev^{-1}(0)$ is equivalent to $K^\perp\subset U^\perp$, and this implies $U\subset K$ by the Hahn-Banach Corollary~\ref{cor:general-Hahn-Banach}. Therefore $K$ is open, i.e., $A$ is Tate.

Conversely, assume that $A$ is Tate. The first condition for
continuity is satisfied with $U$ a c-lattice and $K=U$. Given $a\in
A$, $f\in A^*$, we can choose $K={\rm span}\{a\}$ and $U=\ker f$. Then $a\times K^\perp, U\times f\subset \ev^{-1}(0)$, and therefore $\ev$ is continuous. 
\end{proof}

\begin{remark}
The evaluation map $\ev:A\times A^*\to \bk$ is continuous with respect to each of its variables for any topological vector space $A$. 
\end{remark}

We denote by $B(A\times B,C)$
the space of continuous bilinear maps $A\times B\to C$ equipped with
the {\em compact-open topology}, i.e., the linear topology generated
by the open linear subspaces 
$$
  S_{K\times L,W} = \{f\in B(A\times B,C)\mid f(K\times L)\subset W\}
$$
for $K\subset A$, $L\subset B$ linearly compact and $W\subset C$ open. 

\begin{theorem}[Adjunction\footnote{See~\cite[Proposition A.16]{Hatcher}
for a topological counterpart of this result.}]\label{thm:adjunction}
If $B$ is Tate, then we have a canonical topological isomorphism
$$
  B(A\times B,C) \cong \Hom(A,\Hom(B,C)).
$$
\end{theorem}

\begin{proof}
We have a canonical vector space isomorphism between bilinear maps
$f:A\times B\to C$ and linear maps $\wh f:A\to\Hom_{\rm alg}(B,C)$
related by the formula 
$$
  f(a,b) = \wh f(a)(b). 
$$
It remains to verify that the notions of continuity agree under this
isomorphism. 

(1) If $f:A\times B\to C$ is continuous, then for each $a\in A$ the
map $f_a=\wh f(a):B\to C$ is continuous.
(The map $i_a:B\to C$, $b\mapsto(a,b)$ is
continuous, hence for each open $W\subset C$ the set
$f_a^{-1}(W)=i_a^{-1}(f^{-1}(W))\subset A$ is open.)
Thus the above isomorphism gives a map
$$
  B(A\times B,C) \to \Hom_{\rm alg}(A,\Hom(B,C)).
$$
(2) For $f\in B(A\times B,C)$ the map $\wh f:A\to\Hom(B,C)$ is continuous.
For this, we need to show that for all $L\subset B$ linearly compact
and $W\subset C$ linear open the linear subspace
$$
  \wh f^{-1}(S_{L,W}) = \{a\in A\mid f(a\times L)\subset W\}
$$
is open. Since $f^{-1}(W)\subset A\times B$ is open with respect
to the product topology, there are linear open subsets $U\subset A$
and $V\subset B$ such that $U\times V\subset f^{-1}(W)$.
Since $L$ is linearly compact, $L=(L\cap V)\oplus D$ for
a finite dimensional subspace $D={\rm span}\{b_1,\dots,b_n\}\subset L$. 
For each $i=1,\dots,n$ the map $a\mapsto f(a,b_i)$ is continuous, so the set
$U_i:=\{a\in A\mid f(a,b_i)\in W\}$ is open. It follows that $\wh
f^{-1}(S_{L,W})$ contains the open linear set $U\cap U_1\cap\cdots\cap
U_n$ and is therefore open. 

(3) Conversely, if $\wh f\in\Hom(A,\Hom(B,C))$ then $f\in B(A\times B,C)$. 
For this, given $(a,b)\in A\times B$ with $f(a,b)=c$ and an open
linear subset $W\subset C$ we need to show that $f^{-1}(c+W)$ is an
open neighbourhood of $(a,b)$. So we need to find open linear subsets
$U\subset A$ and $V\subset B$ such that $(a+U)\times(b+V)\subset
f^{-1}(c+W)$, or equivalently,
\begin{equation}\label{eq:open}
  a\times V+U\times b+U\times V\subset f^{-1}(W).
\end{equation}
Since $\wh f(a)\in\Hom(B,C)$, the linear subspace
$$
  V_1:=\wh f(a)^{-1}(W) = \{b'\in B\mid f(a,b')\in W\}
$$
is open. Since $B$ is Tate, there exists an open linearly compact 
subspace $V\subset V_1$, so that $f(a\times V)\subset W$. 
Since $\wh f$ is continuous, the linear subspace
$$
  U_1 := \wh f^{-1}(S_{V,W}) = \{a'\in A\mid f(a'\times V)\subset W\}
$$
is open and satisfies $f(U_1\times V)\subset W$. 
Since the evaluation map $\ev_b:\Hom(B,C)\to C$, $g\mapsto g(b)$ is
continuous, so is the composition $\ev_b\circ\hat f:A\to C$, $a'\mapsto f(a',b)$.
Thus the linear subspace
$$
  U_0 := (\ev_b\circ\hat f)^{-1}(W) = \{a'\in A\mid f(a',b)\in W\}
$$
is open and satisfies $f(U_0\times b)\subset W$. It follows
that~\eqref{eq:open} holds for the open sets $U:=U_0\cap U_1$ and $V$. 

(4) From (2) and (3) we have a linear bijection
$$
  B(A\times B,C) \cong \Hom(A,\Hom(B,C)).
$$
Since the basic linear open sets in $\Hom(A,\Hom(B,C))$ are given for
$K\subset A$, $L\subset B$ linearly compact and $W\subset C$ linear
open by
$$
  \{\wh f\mid \wh f(K)\subset S_{L,W}\} = \{\wh f\mid f(K\times
  L)\subset W\} = \{\wh f\mid f\in S_{K\times L,W}\},
$$
the topologies match and the isomorphism is topological. 
\end{proof}

\section{Tensor products}\label{sec:tensor}

\subsection{Beilinson's three topologies}

Let $A,B$ be linearly topologized vector spaces over a discrete field $\bk$. Beilinson describes in~\cite[\S1.1]{Beilinson} three topologies on the algebraic tensor product $A\otimes B$. A recent and very useful reference is Positselski~\cite{Positselski}. These topologies are also mentioned, albeit with the wrong definition for the $*$ topology and the $^\leftarrow$ topology, in Beilinson-Drinfeld~\cite[\S2.7.7]{BD04}. See also~\cite{Positselski-slides}.  

\begin{definition}\label{def:*top-alex}
A linear subspace $Q\subset A\otimes B$ is open in the {\em $*$
  topology} iff it satisfies the following conditions:
\begin{enumerate}[label=(\roman*)]
\item there exist open linear subspaces $U\subset A$ and $V\subset B$
  such that $U\otimes V\subset Q$;  
\item for each $a\in A$ there exists an open linear subspace $V\subset
  B$ such that $a\otimes V\subset Q$;
\item for each $b\in B$ there exists an open linear subspace $U\subset
  A$ such that $U\otimes b\subset Q$.
\end{enumerate}
\end{definition}

\begin{definition}\label{def:<-top}
A linear subspace $Q\subset A\otimes B$ is open in the {\em $^\leftarrow$
  topology} iff it satisfies the following conditions:
\begin{enumerate}[label=(\roman*)]
\item there exist open linear subspaces $U\subset A$ and $V\subset B$
  such that $U\otimes V\subset Q$;  
\item for each $a\in A$ there exists an open linear subspace $V\subset
  B$ such that $a\otimes V\subset Q$.
\end{enumerate}
\end{definition}

\begin{definition}\label{def:!top}
A linear subspace $Q\subset A\otimes B$ is open in the {\em $!$ topology} iff there exists open linear subspaces $U\subset A$ and $V\subset B$ such that $U\otimes B + A\otimes V\subset Q$.
\end{definition}

We denote $A\otimes^* B$, $A\otimes^\leftarrow B$, $A\otimes^!B$ the tensor product $A\otimes B$ endowed with the respective topologies and $A\hatotimes^*B$, $A\hatotimes^\leftarrow B$, $A\hatotimes^!B$ their completions.\footnote{Here we use the notation of Positselski~\cite{Positselski}.} 

Among the three topologies $*$, $^\leftarrow$, $!$ on $A\otimes B$, the $!$ topology is the coarsest, the $*$ topology is the finest, and the $^\leftarrow$ topology is in between.\footnote{The $^\leftarrow$ topology will only serve as an intermediate between the $*$ and $!$ topologies and will not play any central role, unlike the other two. Note also that the $^\leftarrow$ topology is not symmetric, hence the notation $A\otimes^\rightarrow B=B\otimes^\leftarrow A$ from~\cite[\S1.1]{Beilinson}.} 
In particular, the identity map induces continuous linear maps 
$$
A\otimes^* B\to A\otimes^\leftarrow B\to A\otimes^! B,
$$
and also continuous linear maps  
$$
A\hatotimes^* B\to A\hatotimes^\leftarrow B\to A\hatotimes^! B.
$$

The following lemma is immediate from the definitions. 

\begin{lemma}[{\cite[Lemma~12.9]{Positselski}}]\label{lem:dense}
If $K\subset A$ and $L\subset B$ are dense linear subspaces, then
$K\otimes L\subset A\otimes B$ is dense in $A\otimes^* B$,
$A\otimes^\leftarrow B$, and $A\otimes^! B$. As a consequence
$K\hatotimes^* L=A\hatotimes^*B$, and similarly for
$\hatotimes^\leftarrow$ and $\hatotimes^!$. \qed 
\end{lemma}

As a consequence, we get the following commutativity and associativity properties. 

\begin{proposition} \label{prop:comm-ass}
For linearly topologized vector spaces $A,B,C$ we have canonical isomorphisms
\begin{gather*}
  A\hatotimes^*B\simeq B\hatotimes^*A, \qquad
  (A\hatotimes^*B)\hatotimes^*C \simeq A\hatotimes^*(B\hatotimes^*C), \cr
  A\hatotimes^!B\simeq B\hatotimes^!A, \qquad
  (A\hatotimes^!B)\hatotimes^!C \simeq A\hatotimes^!(B\hatotimes^!C).
\end{gather*}
Moreover, the identity induces a continuous linear map
$$
   A\hatotimes^*(B\hatotimes^!C) \to (A\hatotimes^*B)\hatotimes^!C.
$$
\end{proposition}

\begin{proof}
Everything is obvious except the last assertion. 
In view of Lemma~\ref{lem:dense}, for this we only need to show that 
$$
   \id:A\otimes^*(B\otimes^!C) \to (A\otimes^*B)\otimes^!C
$$
is continuous. In the following $U\subset A$, $V\subset B$ and
$W\subset C$ will always denote open linear subspaces. 
Unravelling the definitions, we find that a linear subspace
$Q\subset A\otimes^*(B\otimes^!C)$ is open iff if satisfies:
\begin{enumerate}[label=(\roman*)]
\item there exist $U\subset A$, $V\subset B$ and $W\subset C$ 
  such that $U\otimes V\otimes C + U\otimes B\otimes W\subset Q$;  
\item for each $a\in A$ there exists $V\subset
  B$ and $W\subset C$ such that $a\otimes V\otimes C + a\otimes
  B\otimes W\subset Q$;
\item for each $b\in B$, $c\in C$ there exists $U\subset
  A$ such that $U\otimes b\otimes c\subset Q$.
\end{enumerate}
On the other hand, $Q\subset (A\otimes^*B)\otimes^!C$ is open iff it satisfies:
\begin{enumerate}
\item there exists $W\subset C$ such that $A\otimes B\otimes W\subset Q$;
\item there exist $U\subset A$ and $V\subset B$
  such that $U\otimes V\otimes C\subset Q$;  
\item for each $a\in A$ there exists $V\subset
  B$ such that $a\otimes V\otimes C\subset Q$;
\item for each $b\in B$ there exists $U\subset
  A$ such that $U\otimes b\otimes C\subset Q$.
\end{enumerate}
We see that conditions (1--4) imply conditions (i--iii). 
\end{proof}

\subsection{The $*$ topology}

The algebraic tensor product $A\otimes B$ (devoid of topology) is endowed with a canonical bilinear map $\pi:A\times B\to A\otimes B$ and is characterized by the following universal property: for any vector space $C$ and any bilinear map $\phi:A\times B\to C$ there is a unique linear map $\phi^\otimes:A\otimes B\to C$ such that $\phi=\phi^\otimes\circ\pi$. 
$$
\xymatrix{
A\times B\ar[r]^-\phi \ar[d]_-\pi & C \\
A\otimes B\ar[ur]_-{\phi^\otimes} &  
}
$$ 

The $*$ topology on $A\otimes B$ is designed to render this correspondence topological. 

\begin{proposition}\label{prop:*topology-universal}
The $*$-topology on $A\otimes B$ is uniquely characterized by any one of the following two conditions. 

(a) It is the finest linear topology such that the canonical bilinear map $\pi:A\times B\to A\otimes B$
is continuous. 

(b) For each linearly topologized space $C$ the assignment $\phi^\otimes\mapsto \phi=\phi^\otimes\circ\pi$ defines a linear bijection
\begin{equation}\label{eq:lin-bilin}
  \pi^*:\Hom(A\otimes^*B,C) \to B(A\times B,C).
\end{equation}
\end{proposition}

\begin{proof}
(a) In view of Lemma~\ref{lem:bil-cont}, the finest linear topology
such that $\pi$ is continuous is characterized by the following
condition: a subspace $E\subset A\otimes B$ is open if and only if
there exist $U\subset A$, $V\subset B$ linear open such that $U\times V\subset \pi^{-1}(E)$, for any $a\in A$ there exists $V\subset B$ linear open such that $a\times V\subset \pi^{-1}(E)$, and for any $b\in B$ there exists $U\subset A$ linear open such that $U\times b\subset\pi^{-1}(E)$. Applying $\pi$ we find exactly the conditions that define the open sets in the $*$ topology. 

(b) It is enough to show that (b) is equivalent to the fact that $A\otimes B$ is endowed with the finest topology such that $\pi$ is continuous. Note that (b) can be rephrased by saying that $\phi$ is continuous if and only if $\phi^\otimes$ is continuous. 

($\Rightarrow$) Let $A\otimes^*B$ be the finest topology that renders $\pi$ continuous, and let $A\otimes'B$ be some topology that satisfies (b). By applying (b) to the commutative diagram 
$$
\xymatrix{
A\times B \ar[r]^-\pi \ar[d]_-\pi & A\otimes' B \\
A\otimes'B \ar[ur]_-{\pi^\otimes=\Id}
}
$$
we infer that the horizontal arrow $\pi$ is continuous because $\Id:A\otimes'B\to A\otimes'B$ is continuous. Thus the topology $A\otimes'B$ is coarser than $A\otimes^*B$. By applying (b) to the commutative diagram 
$$
\xymatrix{
A\times B \ar[r]^-\pi \ar[d]_-\pi & A\otimes^* B \\
A\otimes'B \ar[ur]_-{\pi^\otimes=\Id}
}
$$ 
with the horizontal map $\pi$ continuous, we infer that $\Id:A\otimes'B\to A\otimes^*B$ is continuous, which means that the topology $A\otimes'B$ is finer than $A\otimes^*B$. We conclude that the topologies $A\otimes'B$ and $A\otimes^*B$ are equal. 

($\Leftarrow$) Let now $A\otimes^*B$ be the finest topology such that $\pi$ is continuous. We need to show that, given any bilinear map $\phi:A\times B\to C$ with associated linear map $\phi^\otimes:A\otimes B\to C$, we have that $\phi$ is continuous if and only if $\phi^\otimes$ is continuous. Assuming $\phi^\otimes$ continuous, the continuity of $\phi=\phi^\otimes\circ\pi$ follows from that of $\pi$. Assuming $\phi$ continuous, let $W\subset C$ open. Then for all $c\in C$ we have that $\phi^{-1}(c+W)=\pi^{-1}((\phi^\otimes)^{-1}(c+W))$ is open. By linearity of $\phi^\otimes$, this is equivalent to the fact that $\pi^{-1}(d+(\phi^\otimes)^{-1}(W))$ is open for all $d\in A\otimes B$, which means that $(\phi^\otimes)^{-1}(W)$ is open by the definition of the finest topology $A\otimes^*B$. This in turn amounts to the continuity of $\phi$. 
\end{proof}

\begin{proposition}\label{prop:*otimes-lincompact}
(a) If $A$ and $B$ are discrete, then $A\hatotimes^*B=A\otimes^*B$ is discrete.\\
(b) If $A$ and $B$ are linearly compact, then $A\hatotimes^*B$ is linearly compact. 
\end{proposition}

\begin{proof}
Part (a) holds because by definition $A\otimes^*B$ is discrete, hence complete, so that $A\hatotimes^*B=A\otimes^*B$. 
Part (b) is Lemma~\ref{lem:A*otimesBlinearlycompact} below.
\end{proof}

Note that $A\otimes^*B$ is in general not complete, hence not linearly compact. For example $\bk[[t]]\otimes^*\bk[[s]]$ is not complete, but $\bk[[t]\hatotimes^*\bk[[s]]=\bk[[t,s]]$ is complete and linearly compact. 

We now discuss the extent to which the linear
bijection~\eqref{eq:lin-bilin} can be made topological, where $B(A,B)$
is equipped with the compact-open topology defined
in~\S\ref{ss:bilin}. Given a bilinear map $\phi:A\times B\to C$ consider the diagram 
$$
\xymatrix{
A\times B\ar[rr]^-\phi \ar[d]_-\pi \ar@/_2pc/[dd]_-{\hat\pi} & & C \\
A\otimes B\ar[urr]_-{\phi^\otimes} \ar[d]_-\iota & &  \\
A\hatotimes^*B\ar@/_/[uurr]_-{\hat\phi^\otimes} & & 
}
$$ 
in which $\iota$ is the canonical injection into the completion and $\hat\pi=\iota\circ\pi$. 

\begin{proposition}\label{prop:*topological}
Assume $C$ to be complete. Then the map 
$$
\hat\pi^*:\Hom(A\hatotimes^*B,C)\to B(A\times B,C),\qquad \hat\phi^\otimes\mapsto \hat\phi^\otimes\circ\hat\pi
$$
is a continuous linear bijection. 
Moreover, $\hat\pi^*$ is a topological isomorphism in either of the following two cases: 

(a) $A,B$ are linearly compact and $C$ is discrete.

(b) $A,B$ are discrete and $C$ is linearly compact. 
\end{proposition}

\begin{proof}
That $\hat\pi^*$ is a linear bijection follows from Proposition~\ref{prop:*topology-universal}(b), in view of the fact that $\Hom(A\otimes^*B,C)=\Hom(A\hatotimes^*B,C)$ as vector spaces. Indeed, any map defined on the dense subspace $A\otimes^*B\subset A\hatotimes^*B$ extends uniquely to the completion, with values in the completion $\widehat C=C$. (Note however that the topologies on $\Hom(A\otimes^*B,C)$ and $\Hom(A\hatotimes^*B,C)$ differ in general, since there is no canonical correspondence between linearly compact spaces in $A\otimes^*B$ and in $A\hatotimes^*B$.)

We prove that $\hat\pi^*$ is continuous. A basis of open sets for $B(A\times B,C)$ is given by $S_{K\times L,W}=\{\phi\mid \phi(K\times L)\subset W\}$, with $K\subset A$, $L\subset B$ linearly compact and $W\subset C$ linear open. A basis for $\Hom(A\hatotimes^*B,C)$ is given by $S_{E,W}=\{\hat\phi^\otimes\mid \hat\phi^\otimes(E)\subset W\}$, with $E\subset A\hatotimes^*B$ linearly compact and $W\subset C$ linear open. We have that $(\hat\pi^*)^{-1}(S_{K\times L,W})=\{\hat\phi^\otimes\mid \hat\phi^\otimes\circ\hat\pi(K\times L)\subset W\}$. Now $K\hatotimes^*L$ is linearly compact by Proposition~\ref{prop:*otimes-lincompact}, so that $(\hat\pi^*)^{-1}(S_{K\times L,W})$ contains the open set $S_{K\hatotimes^*L,W}$ and is therefore open.

In case (a) both the source and the target of $\hat\pi^*$ are discrete, and therefore $\hat\pi^*$ is a topological isomorphism. 

In case (b) the source is linearly compact by Proposition~\ref{prop:*otimes-lincompact} and Lemma~\ref{lem:Hom1}(b). Since $\hat\pi^*$ is linear bijective and continuous, we infer that it is a topological isomorphism~\cite[II.27.8]{Lefschetz-book}. 
\end{proof}

\subsection{The $!$ topology}

\begin{proposition}\label{prop:characterization!topologyTate}
If $A$ and $B$ are Tate, then the $!$ topology is the coarsest topology on $A\otimes B$ such that the linear map 
$$
\beta^\otimes:A\otimes B\to B(A^*\times B^*,\bk),\qquad a\otimes b \mapsto \big( (\varphi,\psi)\mapsto \varphi(a)\psi(b)\big)
$$
is continuous.
\end{proposition}

This will follow from the more general statement below. Given a subspace $E\subset A^*$, let 
$$
E^{\perp_A}=\{a\in A\mid \ev(a\otimes E)=0\}.
$$ 

\begin{lemma}\label{lem:characterization!topology}
For linearly topologized spaces $A,B$, 
the coarsest linear topology on $A\otimes B$ such that the linear map 
$$
\beta^\otimes:A\otimes B\to B(A^*\times B^*,\bk),\qquad a\otimes b \mapsto \big( (\varphi,\psi)\mapsto \varphi(a)\psi(b)\big)
$$
is continuous has a basis of open sets consisting of $\cK^{\perp_A}\otimes B + A\otimes \cL^{\perp_B}$, where $\cK\subset A^*$ and $\cL\subset B^*$ are linearly compact.
\end{lemma}

\begin{proof}
To show that $\beta^\otimes$ is well-defined (i.e., it lands in
$B(A^*\times B^*,\bk)$), we consider the associated bilinear map
$\beta:A\times B\to B(A^*\times B^*,\bk)$ given by $\beta(a,b)=\big(
(\varphi,\psi)\mapsto \varphi(a)\psi(b)\big)$ and prove that this is
well-defined. We need to show that $\beta(a,b)$ is a continuous
bilinear map. By Lemma~\ref{lem:bil-cont}, this amounts to the conditions that $\beta(a,b)^{-1}(0)$ contains $K^\perp\times L^\perp$ for some $K\subset A$, $L\subset B$ linearly compact, for any $\varphi\in A^*$ there exists $L_\varphi\subset B$ linearly compact such that $\varphi\times L_\varphi^\perp\subset \beta(a,b)^{-1}(0)$, and for any $\psi\in B^*$ there exists $K_\psi\subset A$ linearly compact such that $K_\psi^\perp\times\psi\subset \beta(a,b)^{-1}(0)$. These conditions are satisfied with $K=K_\psi\ni a$ and $L=L_\varphi\ni b$ linearly compact.

A basis of neighborhoods of $0$ in $B(A^*\times B^*,\bk)$ is given by $S_{\cK\times \cL,0}=\{\gamma:A^*\times B^*\to\bk\mid \gamma(\cK\times\cL)=0\}$, with $\cK\subset A^*$, $\cL\subset B^*$ linearly compact. One then checks directly that 
$(\beta^\otimes)^{-1}(S_{\cK\times\cL,0}) = \cK^{\perp_A}\otimes B + A\otimes \cL^{\perp_B}$. 
\end{proof}

\begin{proof}[Proof of Proposition~\ref{prop:characterization!topologyTate}]
If $A$ is Tate then $A^{**}=A$, a basis of c-lattices for $A^*$ is
given by $\{U^\perp\}$, where $U$ ranges over the c-lattices of $A$,
and $(U^\perp)^{\perp_A}=U^{\perp\perp}=U$ for any c-lattice. Since
any linearly compact subspace $\cK\subset A^*$ is contained in a
c-lattice, the topology from
Lemma~\ref{lem:characterization!topology} with basis of open
neighborhoods $\cK^{\perp_A}\otimes B + A\otimes\cL^{\perp_B}$ with
$\cK\subset A^*$, $\cL\subset B^*$ linearly compact, can be
equivalently described as having a basis of open neighborhoods of zero
given by $(U^{\perp})^{\perp_A}\otimes B  + A\otimes
(V^\perp)^{\perp_B} = U\otimes B + A \otimes V$, where $U\subset A$
and $V\subset B$ are c-lattices. Since the c-lattices in a Tate vector
space form a neighborhood basis of zero, we infer that this is the ! topology. 
\end{proof}

\begin{proposition}\label{prop:!otimes-Hom}
For $A,B$ Tate we have canonical isomorphisms
$$
A\hatotimes^! B\simeq \Hom(B^*,A)\simeq \Hom(A^*,B).
$$
\end{proposition}

\begin{proof}
Since Tate spaces are isomorphic to their biduals it is enough to prove that 
\begin{equation} \label{eq:A*!otimesB}
A^*\hatotimes^!B\simeq \Hom(A,B).
\end{equation}
Recall from~\S\ref{ss:comp-disc-fin} the subspace
$\Hom_\fin(A,B)\subset \Hom(A,B)$ of finite rank homomorphisms. This is a dense subspace (Lemma~\ref{lem:Hom1}(e) with $A$ Tate) of the space $\Hom(A,B)$ which is complete (Lemma~\ref{lem:Hom2}), so that $\widehat{\Hom_\fin(A,B)}=\Hom(A,B)$. 

On the other hand we have a canonical linear bijective correspondence 
$$
\Ev:A^*\otimes B\simeq \Hom_\fin(A,B), \qquad \varphi\otimes b\mapsto (a\mapsto \varphi(a)b).
$$ 
In order to prove~\eqref{eq:A*!otimesB} it is therefore enough to show that, under the identification $\Ev$, the $!$ topology on $A^*\otimes B$ corresponds to the topology induced on $\Hom_\fin(A,B)$ from $\Hom(A,B)$. 

A neighborhood basis of $0$ in $\Hom(A,B)$ is given by
$S_{U,V}=\{f:A\to B\mid f(U)\subset V\}$ for $U\subset A$ and
$V\subset B$ c-lattices. A neighborhood basis of $0$ in $A^*\otimes^! B$ is given by $U^\perp\otimes B + A\otimes V$ for $U\subset A$ and $V\subset B$ c-lattices. Let $f\in \Hom_\fin(A,B)$. We claim that $f(U)\subset V$ if and only if $f\in\Ev(U^\perp\otimes B + A\otimes V)$. 

($\Leftarrow$) This is clear by definition of $\Ev$. 

($\Rightarrow$) Let $D=f(A)$. Since $f\in \Hom_\fin(A,B)$ we infer that $\ker f$ is open, hence $\ker f\cap U$ has finite codimension in $U$. We choose $D^U$ a finite dimensional complement of $\ker f\cap U$ in $U$, $H$ a (discrete) complement of $U\cap \ker f$ in $\ker f$, and $D'$ a finite dimensional complement of $D^U\oplus \ker f$ in $A$. Then $f:D^U\oplus D'\stackrel\simeq\longrightarrow D$. Choose bases $e_1,\dots,e_k$ of $D^U$, $f_1,\dots,f_\ell$ of $D'$, with $f$-images $v_1,\dots,v_k\in V$ and $b_1,\dots,b_\ell\in B$. Extend the map $e_i\mapsto 1$ by $0$ on $\ker f$ and on the other basis elements of $D^U\oplus D'$, and denote this extension $e_i^*$ for $i=1,\dots,k$. Define $f_j^*$ similarly for $j=1,\dots,\ell$. 
We then have $f=\Ev(z)$ with $z= \sum_{j=1}^\ell f_j^*\otimes b_j +
\sum_{i=1}^k e_i^*\otimes v_i \in U^\perp\otimes B + A^*\otimes V$. This concludes the proof.
\end{proof}

\begin{remark}[Tensor products of Banach spaces]\label{rem:tensor-Banach}
It is pointed out by Positselski~\cite{Positselski-slides} that the
$*$ and $!$ topologies above are analogues of the projective and
injective topologies on tensor products of real Banach spaces $X,Y$.
Let us recall the definitions from~\cite[\S2 and \S3]{Ryan}.

(i) The {\em projective norm} of $u\in X\otimes Y$ is 
$$
  \pi(u) := \inf\Bigl\{\sum_{i=1}^n\|x_i\|\,\|y_i\|\;\Bigl|\;
   u=\sum_{i=1}^nx_i\otimes y_i \Bigr\}.
$$
In view of Proposition~\ref{prop:*topology-universal}, this is
analogous to the $*$ topology. An analogue of our Proposition~\ref{prop:*topological} 
for the projective norm on Banach spaces is given in~\cite[Theorem 2.9]{Ryan}.

(ii) The {\em injective norm} of $u=\sum_{i=1}^nx_i\otimes y_i\in
X\otimes Y$ is 
$$
  \eps(u) := \sup\Bigl\{\bigl|\sum_{i=1}^n\varphi(x_i)\psi(y_i)\bigr|\;\Bigl|\;
    \varphi\in X^*,\psi\in Y^*,\|\varphi\|=\|\psi\|=1\Bigr\}.
$$
In view of Proposition~\ref{prop:characterization!topologyTate}, this
is analogous to the $!$ topology on tensor products of Tate vector
spaces. Away from Tate spaces, the correct analogue of the injective
tensor product appears to be the new linear topology from
Lemma~\ref{lem:characterization!topology}. As such, it may be
important \emph{per se}. 
\end{remark}

\subsection{Linearly compact spaces and discrete spaces} \label{sec:lincomp-discr-tensor}

The following result is stated without proof in~\cite[Remark~12.1]{Positselski} and~\cite[\S1.1]{Beilinson}.

\begin{lemma}\label{lem:A*otimesBlinearlycompact}
Assume $A$ and $B$ are linearly compact. Then all three topologies on $A\otimes B$ coincide, and 
$$
A\hatotimes^*B=A\hatotimes^\leftarrow B=A\hatotimes^!B=\Hom(B^*,A)
$$
is linearly compact. 
\end{lemma}

\begin{proof}
By Proposition~\ref{prop:!otimes-Hom} we have
$A\hatotimes^!B=\Hom(B^*,A)$, which is linearly compact by
Lemma~\ref{lem:Hom1}(a) because $B^*$ is discrete. 
It remains to show that the $*$ and $!$ topologies coincide, which
amounts to proving that any linear subspace $E\subset A\otimes B$ that
is open in the $*$ topology is also open in the $!$ topology. By
assumption there exist $U\subset A$, $V\subset B$ linear open such
that $U\otimes V\subset E$, for any $a\in A$ there exists $V_a\subset
B$ linear open such that $a\otimes V_a\subset E$, and for any $b\in B$
there exists $U_b\subset A$ linear open such that $U_b\otimes b\subset
E$. Since $A$ and $B$ are linearly compact we find finite dimensional
complements of $U$ in $A$ and of $V$ in $B$, say $A=U\oplus{\rm
  span}\{a_1,\dots,a_k\}$ and $B=V\oplus{\rm
  span}\{b_1,\dots,b_\ell\}$. Let $U'=U\cap U_{b_1}\cap\dots\cap
U_{b_\ell}$ and $V'=V\cap V_{a_1}\cap\dots\cap V_{a_k}$. These are
open linear subspaces such that $U'\otimes B+A\otimes V'\subset E$,
which means that $E$ is open in the $!$ topology. 
\end{proof}

\begin{lemma} \label{lem:tensor_discrete}
Assume that $A$ and $B$ are discrete. Then all three topologies on $A\otimes B$ coincide, and 
$$
A\otimes B = A\hatotimes^*B=A\hatotimes^\leftarrow B=A\hatotimes^!B  \simeq \Hom(B^*,A) \simeq \Hom(A^*,B)
$$
is discrete.
\end{lemma}

\begin{proof}
That all three topologies on $A\otimes B$ are discrete and coincide
follows immediately from the definitions. This proves the first three
equalities. The last equality follows from Proposition~\ref{prop:!otimes-Hom}.
\end{proof}

\begin{proposition} \label{prop:duality*!}
Assume that $A$, $B$ are both linearly compact, or both discrete. Then we have topological isomorphisms
$$
(A\hatotimes^* B)^*\simeq A^*\hatotimes^! B^*,\qquad (A\hatotimes^! B)^*\simeq A^*\hatotimes^*B^*.
$$
\end{proposition}

\begin{proof}
The first claim follows from the sequence of topological isomorphisms 
$$
(A\hatotimes^* B)^*\simeq B(A\times B,\bk)\simeq \Hom(A,B^*)\simeq A^*\hatotimes^! B^*.
$$
Here the first isomorphism is Proposition~\ref{prop:*topological}, the second one Theorem~\ref{thm:adjunction}, and the third one Proposition~\ref{prop:!otimes-Hom}. 
The second claim follows from the first one by duality.
\end{proof}

The following example shows that the completed $*$ and $!$ tensor products may be different.

\begin{example}[\cite{Positselski-slides}]
Let $A$ be complete and $B$ discrete with basis $\{b_x\mid x\in
X\}$. Denote 
$$
  A[X] := \bigoplus_{x\in X}A,\qquad 
  A[[X]] := \{(a_x)\in\prod_{x\in X}A \text{ zero-convergent}\},
$$
where $(a_x)$ is called {\em zero-convergent} if $\{x\mid a_x\notin
U\}$ is finite for each open linear subspace $U\subset A$. Then
$$
  A\hatotimes^*B = A\otimes^*B = A[X]\quad\text{and}\quad A\hatotimes^!B = A[[X]].
$$
For example, for $A=\bk[[s]]$ (linearly compact) and $B=\bk[t]$ we get
$$
  A\hatotimes^*B = \bk[[s]][t] \subsetneq A[[\N_0]] = A\hatotimes^!B
$$
because $\sum_{n\geq 0}s^nt^n\in A[[\N_0]]\setminus \bk[[s]][t]$. 
\end{example}

\subsection{Ind-linearly compact and pro-discrete spaces} \label{sec:ind-pro}

In order to upgrade Proposition~\ref{prop:duality*!} to Tate spaces $A,B$ we need to work outside the category of Tate spaces. The fundamental reason is that, if $A$ and $B$ are Tate, then $A\hatotimes^! B\simeq \Hom(A^*,B)$ need not be Tate, cf.~Lemma~\ref{lem:Hom1}(c-d), and $A\hatotimes^* B$ need not be Tate either, cf.~\cite[Example~13.1(1)]{Positselski}.
Following Esposito-Penkov~\cite{Esposito-Penkov22,Esposito-Penkov23}, we consider the following categories of ind-linearly compact spaces and pro-discrete spaces.\footnote{Esposito-Penkov work under various countability assumptions, which we do not impose. We discuss specific consequences of countability in~\S\ref{sec:count}.} We work in the category of complete
linearly topologized Hausdorff spaces, with limits and colimits over arbitrary index sets.

\begin{definition}
A linearly topologized Hausdorff space $A$ is called
\begin{itemize}
\item {\em ind-linearly compact} if $A=\wh\colim A_i$ for a direct system
  of linearly compact spaces $A_i$;
\item {\em pro-discrete} if $A=\lim A_i$ for an inverse system
  of discrete spaces $A_i$. 
\end{itemize}  
\end{definition}
Here $\wh\colim A_i$ denotes the completion of the colimit as in~\S\ref{sec:completeness} (this is
the correct notion of a colimit in the category of complete linearly
topologized spaces).
These notions taken from~\cite{Esposito-Penkov22} (minus the
countability hypotheses) correspond to our earlier notions as follows: 
$$
\xymatrix
@C=50pt
{
\text{ind-linearly compact} \ar@{<->}[r] \ar@{<->}[d] &
\text{linearly compactly generated} \ar@{<->}[d] \\
\text{pro-discrete} \ar@{<->}[r] & \text{complete} \\
}
$$
Here the vertical arrows are duality, where the left one
is~\cite{Esposito-Penkov22} while the right one is our
Lemma~\ref{lem:Hom2}. 
The horizontal arrows follow from Corollary~\ref{cor:lim-colim-duality}.

\subsection{Tensor products on the categories $\cI$ and $\cP$} \label{sec:tensor-ind-pro}

We denote the categories of ind-linearly compact and pro-discrete
spaces (with continuous linear maps as morphisms) by $\cI$ and $\cP$,
respectively. Let $Tate$ be the category of Tate vector spaces.

\begin{lemma} \label{lem:IPTate} 
We have 
$$
\cI\cap \cP\supseteq Tate.
$$
\end{lemma}

\begin{proof}
Given a Tate vector space $V$ we present it as $V=L\oplus D$ with $L$ linearly compact and $D$ discrete (Lemma~\ref{lem:top-splitting}). To prove that $V\in\cI$ we note that $D=\colim F$ with $F\subset D$ finite dimensional, so that $V=\colim (L\oplus F)$ with $L\oplus F$ linearly compact. To prove that $V\in\cP$ we note that, by completeness of $L$, we have $L=\lim L/U$ with $U\subset L$ open, hence $L/U$ discrete, and therefore $V=\lim (L/U\oplus D)$ with $L/U\oplus D$ discrete. 
\end{proof}
 
\begin{remark}
It is not know to us at the time of writing whether the inclusion from Lemma~\ref{lem:IPTate} is an equality. We prove an equality under countability assumptions in Proposition~\ref{prop:IPTateomega}.
\end{remark}

The next result can be interpreted as a generalization of the Duality Theorem~\ref{thm:Tate-duality} for Tate vector spaces.

\begin{corollary}[{\cite{Esposito-Penkov23}}] \label{cor:dualityIP} Topological duality $V\mapsto V^*$ induces an involutive equivalence of categories $\cI\simeq \cP$. 
\end{corollary}

\begin{proof}
This is a straightforward consequence of Corollary~\ref{cor:lim-colim-duality}.
\end{proof}

\begin{proposition}[{cf.~\cite[Proposition 2.3]{Esposito-Penkov23}}] \label{prop:duality_intertwines}
(a) The category $\cI$ is closed under the tensor product
  $\hatotimes^*$. \\ 
(b) The category $\cP$ is closed under the tensor product
  $\hatotimes^!$. \\ 
(c) The duality between $\cI$ and $\cP$ intertwines $\hatotimes^*$ and
  $\hatotimes^!$. \\
\end{proposition}

\begin{proof}
(a) Let $A=\colim A_i$ and $B=\colim B_j$ for direct systems of
linearly compact spaces $A_i,B_j$, so that $\wh A$ and $\wh B$ belong
to $\cI$. The algebraic tensor products satisfy
$$
  A\otimes B = \colim_{i,j}A_i\otimes B_j. 
$$
We claim that with topologies we have
$$
  A\otimes^*B = \colim_{i,j}A_i\otimes^*B_j. 
$$
To see this, recall that a subset $Q\subset A\otimes^*B$ is open iff
it satisfies the following conditions:
\begin{enumerate}[label=(\roman*)]
\item there exist open linear subspaces $U\subset A$ and $V\subset B$
  such that $U\otimes V\subset Q$;  
\item for each $a\in A$ there exists an open linear subspace $V\subset
  B$ such that $a\otimes V\subset Q$;
\item for each $b\in B$ there exists an open linear subspace $U\subset
  A$ such that $U\otimes b\subset Q$.
\end{enumerate}
Now the basic linear open subsets of $A=\colim A_i$ have the form
$U=\bigoplus U_i$ with $U_i\subset A_i$ linear open, and similarly
$V=\bigoplus V_j$ with $V_j\subset B_j$ linear open. Therefore, the
conditions become:
\begin{enumerate}[label=(\roman*)]
\item there exist open linear subspaces $U_i\subset A_i$ and $V_j\subset B_j$
  such that $\bigoplus_{i,j}U_i\otimes V_j\subset Q$;  
\item for each $a_i\in A_i$ there exist open linear subspaces $V_j\subset
  B_j$ such that $\bigoplus_ja_i\otimes V_j\subset Q$;
\item for each $j\in B_j$ there exist open linear subspaces $U_i\subset
  A_i$ such that $\bigoplus_iU_i\otimes b_j\subset Q$.
\end{enumerate}
On the other hand, the basic linear open subsets of
$\colim_{i,j}A_i\otimes^*B_j$ (with respect to the colimit topology)
have the form $Q=\bigoplus_{i,j}Q_{ij}$ with $Q_{ij}\subset
A_i\otimes^*B_j$ linear open. Now openness of $Q_{ij}$ means:
\begin{enumerate}[label=(\roman*)]
\item there exist open linear subspaces $U_i\subset A_i$ and $V_j\subset B_j$
  such that $U_i\otimes V_j\subset Q_{ij}$;  
\item for each $a_i\in A_i$ there exists an open linear subspace $V_j\subset
  B_j$ such that $a_i\otimes V_j\subset Q_{ij}$;
\item for each $j\in B_j$ there exists and open linear subspace $U_i\subset
  A_i$ such that $U_i\otimes b_j\subset Q_{ij}$.
\end{enumerate}
We see that the topologies match, so the claim is proved. Taking
completions of both sides of the claim we get
$$
  \wh A\hatotimes^*\wh B = A\hatotimes^*B = \wh\colim_{i,j}A_i\otimes^*B_j =
  \wh\colim_{i,j}A_i\hatotimes^*B_j \in \cI.
$$
Here the first isomorphism is Lemma~\ref{lem:dense} applied to the
dense subspaces $A\subset \wh A$ and $B\subset \wh B$.
The second isomorphism is the claim above. 
The third isomorphism uses the easy observation that, for a direct system $C_i$ of linearly topologized vector spaces, $\colim_i C_i$ is dense in $\colim_i \wh C_i$, hence their completions are the same (Lemma~\ref{lem:completion-dense}).
The last space belongs to $\cI$ because $A_i\hatotimes^*B_j$ is linearly compact by
Lemma~\ref{lem:A*otimesBlinearlycompact}.

(b) Let $A=\lim A_i$ and $B=\lim B_j$ for inverse systems of discrete
spaces $A_i,B_j$. 
Recall that the basic linear open sets in $A\otimes^!B$
are of the form $Q=U\otimes B+A\otimes V$ with $U\subset A$, $V\subset B$
linear open. Since the $A_i$ are discrete, the basic linear open sets
in $A=\lim A_i$ are finite intersections of the sets
$\prod_{i'\neq i}A_{i'}$ (where the $i$-the component is zero),
and similarly in $B$. So the basic linear open sets in $A\otimes^!B$
are finite intersections of the sets 
$$
  Q_{ij} = \prod_{i'\neq i}A_{i'}\otimes B + A\otimes\prod_{j'\neq
    j}A_{j'},\qquad i\in I,\,j\in J. 
$$
Since $A\otimes^!B/Q_{ij} = A_i\otimes B_j$ with the discrete
topology, by definition of the completion we get   
$$
   A\hatotimes^!B = \lim_{i,j}(A\otimes^!B/Q_{ij}) =
   \lim_{i,j}A_i\otimes B_j \in \cP.
$$
(c) Let $A=\colim A_i$ and $B=\colim B_j$ for direct systems of
linearly compact spaces $A_i,B_j$, so that $\wh A$ and $\wh B$ belong
to $\cI$. Recall from the proof of (a) that
$$
  \wh A\hatotimes^*\wh B = \wh\colim_{i,j}A_i\hatotimes^*B_j
$$
with $A_i\hatotimes^*B_j$ linearly compact. Dualizing this, we compute
\begin{align*}
  (\wh A\hatotimes^*\wh B)^*
  &= (\wh\colim_{i,j}A_i\hatotimes^*B_j)^* 
  = (\colim_{i,j}A_i\hatotimes^*B_j)^* \cr
  &= \lim_{i,j}(A_i\hatotimes^*B_j)^* 
  = \lim_{i,j}A_i^*\otimes B_j^* \cr
  &= (\lim_iA_i^*)\hatotimes^!(\lim_jB_j^*) 
  = A^*\hatotimes^!B^*.
\end{align*}
Here the third equality follows from Lemma~\ref{lem:hom-lim}(b), applied
with the linearly compact (hence Tate) space $A_i\hatotimes^*B_j$ as
source and $\bk$ as target.
The fourth equality follows from Proposition~\ref{prop:duality*!} and Lemma~\ref{lem:tensor_discrete},
the fifth one from the proof of (b) applied to the discrete spaces
$A_i^*,B_j^*$,
and the last one again by Lemma~\ref{lem:hom-lim}(b). 
This proves the first assertion. 

For the second assertion, let $A=\lim A_i$ and $B=\lim B_j$ for inverse systems of discrete spaces $A_i$ and $B_j$, so that $A,B\in \cP$.
By Corollary~\ref{cor:dualityIP} we have $A=A^{**}$ and $B=B^{**}$, so that we can apply the first assertion to $A^*$ and $B^*$, which belong to $\cI$, to obtain
$$
  (A^*\hatotimes^*B^*)^*= A\hatotimes^!B.
$$
Dualizing both sides and using involutivity in Corollary~\ref{cor:dualityIP} we get
$$
(A \hatotimes^! B)^*= A^*\hatotimes^* B^*. 
$$
\end{proof}

\section{Countability conditions}\label{sec:count}

The subcategories $Tate\subset\Top^c\subset\Top$ of Tate resp.~complete
linearly topologized vector spaces have several shortcomings:
they are not closed under cokernels and colimits, closed subspaces
need not have complements, short exact sequences need not split, 
and bijective morphisms need not be isomorphisms. 
In this section and Appendix~\ref{sec:HBseq} we will see that all of these get remedied under
suitable countability assumptions. Since these countability
assumptions are satisfied for Rabinowitz Floer homology, some results
from this section will be useful in~\S\ref{sec:RFH}. 

\subsection{Countability hypotheses} \label{sec:countability_hypotheses}

We will consider three countability hypotheses: 
\begin{enumerate}
\item a linearly topologized vector space $A$ has
  \emph{countable basis} if it has a countable basis of linear
  neighborhoods of $0$; 
\item $A$ has {\em countable dimension} if it has
  countable dimension as a $\bk$-vector space;
\item a direct or inverse system $A_i$ is {\em countably indexed} if
  it has the index set $(\N,\leq)$.  
\end{enumerate}
A space $A$ has countable basis iff it has a nested countable basis
$A\supset U_1\supset U_2 \supset \dots\supset \{0\}$ with $\bigcap_i
U_i = \{0\}$. The completion $\wh A=\lim_i A/U_i$ then also has
countable basis.   

A space $A$ has countable dimension iff there exist nested finite
dimensional subspaces $F_1\subset F_2\subset\cdots$ such that
$\colim_iF_i=\bigcup_iF_i=A$.

These two conditions are dual to each other under additional restrictions: the dual of a linearly compact space with countable basis is discrete of countable dimension, and vice-versa. 

If $A_i$ is a countably indexed inverse system of spaces with countable
bases, then $\lim A_i$ has countable basis.
If $A_i$ is a countably indexed direct system of spaces of countable
dimension, then $\colim A_i$ has countable dimension.
These are direct consequences of the fact that the set of finite
subsets of $\N$ is countable.

These countability hypotheses give rise to nice categories.

\begin{definition}
A linearly topologized vector space $A$ is called
\begin{itemize}
\item {\em ind$^\omega$-linearly compact} if $A=\wh\colim A_i$ for a
  countably indexed direct system
  of linearly compact spaces $A_i$ with countable bases;
\item {\em pro$^\omega$-discrete} if $A=\lim A_i$ for a countably indexed inverse system
  of discrete spaces $A_i$ of countable dimension. 
\end{itemize}  
We denote $\cI^\omega$ the category of ind$^\omega$-linearly compact
spaces, and $\cP^\omega$ the category of pro$^\omega$-discrete spaces.\footnote{  
These correspond to the categories $\cI$ and $\cP$ in~\cite{Esposito-Penkov22,Esposito-Penkov23}.}
\end{definition}

Many important results hold in $\cI^\omega$ and $\cP^\omega$ (see also Appendix~\ref{sec:HBseq}):
\begin{itemize}
\item open mapping theorem in $\cI^\omega$ and $\cP^\omega$~\cite[Lemma
  2.7]{Esposito-Penkov22};
\item topological splitting, see Corollary~\ref{cor:exists-closed-complement}
  and~\cite[Lemma 2.8]{Esposito-Penkov22}.   
\item morphisms in $\cI^\omega$ and $\cP^\omega$ have kernels and cokernels,
  see~\cite[\S2.2]{Esposito-Penkov22} and~\cite[Lemma 1.2 and
    Proposition 1.4]{Positselski}; 
\end{itemize}

\begin{remark}[{\cite[Lemma 2.3]{Esposito-Penkov22}}]
If $L$ is linearly compact with countable basis $U_1\supset
U_2\supset\cdots$ such that $\dim(L/U_i)\to\infty$ as $i\to\infty$,
then $\dim L$ is uncountable, so with the discrete
topology $L$ does not belong to $\cP^\omega$. 
In particular, the main counterexample in Appendix~\ref{sec:HBseq},
the space $\bk[[t]]$ with the discrete topology, does not belong to $\cP^\omega$.
\end{remark}

The following results are the ``countable" counterparts of
Corollary~\ref{cor:dualityIP} and Proposition~\ref{prop:duality_intertwines}.  

\begin{proposition}[{\cite[Proposition 2.6]{Esposito-Penkov22}}] \label{cor:dualityIP_omega} 
Topological duality $V\mapsto V^*$ induces an involutive equivalence
of categories $\cI^\omega\simeq \cP^\omega$. \qed 
\end{proposition}

\begin{proposition}[{\cite[Proposition 2.3]{Esposito-Penkov23}}] \label{prop:duality_intertwines_omega}
\qquad 

(a) The category $\cI^\omega$ is closed under the tensor product
  $\hatotimes^*$. \\ 
(b) The category $\cP^\omega$ is closed under the tensor product
  $\hatotimes^!$. \\ 
(c) The duality between $\cI^\omega$ and $\cP^\omega$ intertwines $\hatotimes^*$ and
  $\hatotimes^!$. \qed
\end{proposition}

Let $Tate^\omega$ be the category of Tate vector spaces that satisfy the following two conditions: 
\begin{itemize}
\item they have countable basis;
\item they admit a c-lattice of countable codimension. 
\end{itemize}
It is easy to see that a linearly topologized vector space
belongs to $Tate^\omega$ if and only if it can be written as a direct
sum $L\oplus D$, where $L$ is linearly compact with countable basis
and $D$ is discrete of countable dimension.  

\begin{proposition} \label{prop:IPTateomega}
We have 
$$
\cI^\omega\cap \cP^\omega = Tate^\omega.
$$
\end{proposition}

\begin{proof}
The inclusion $\supset$ is proved as in Lemma~\ref{lem:IPTate}. 
For the converse inclusion $\subset$, consider $A\in\cI^\om\cap\cP^\om$.
Then $A=\wh\colim A_i$ for a countably indexed direct system of
linearly compact spaces $A_i$ with countable bases. After replacing
$A_i$ by $A_1\oplus\cdots\oplus A_i$, we may assume that the maps
$f_{ji}:A_i\to A_j$ in the direct system are injective. 
According to~\cite[Lemma 2.5(i)]{Esposito-Penkov23}, $A$ belongs to
$\cP^\om$ if and only if $\coker f_{ji}$ is finite dimensional for all
sufficiently large $i\leq j$. Then Corollary~\ref{cor:limit-colimit-Tate}(b)
implies that $A$ is Tate, and the proof shows that it actually belongs
to $Tate^\om$. 
\end{proof}

\begin{remark}
Esposito-Penkov claim in~\cite[Definition~2.4]{Esposito-Penkov23} the equality $Tate^\omega=\cI^\omega\cap \cP^\omega$ but, as far as we can tell, this is not proved in \emph{loc.~cit.}
\end{remark}

\subsection{Completeness}

The following observations about completeness are relevant in our context: 
\begin{itemize}
\item it is preserved under limits and direct sums~\cite[Lemma 1.2]{Positselski};
\item it is in general {\em not} preserved under cokernels and
  colimits~\cite[Theorem 2.5]{Positselski};
\item it is preserved under cokernels and countably indexed
  colimits of spaces with countable bases~\cite[Proposition 1.4]{Positselski}.
\end{itemize}
The following corollary is used in the proof
of~\cite[Proposition 2.3(i)]{Esposito-Penkov23}: 

\begin{corollary}
For a countably indexed direct system of spaces $V_i$ with countable
bases, their completions satisfy
$$
  \colim\wh V_i = \wh{\colim V_i}. 
$$
\end{corollary}

\begin{proof}
Since $\colim V_i$ is dense in $\colim \wh V_i$, their completions are
equal (Lemma~\ref{lem:completion-dense}). The conclusion then follows
from the fact mentioned above that $\colim \wh V_i$ is
complete~\cite[Proposition 1.4]{Positselski}. 
\end{proof}

\begin{lemma} \label{lem:colim_countable}
Let $A=\colim A_i$ for a countably indexed direct system
$(A_i,f_{ji})$ of spaces with countable bases. 

(a) $A$ is isomorphic to the colimit of such a system in which all maps are
injective.

(b) $A$ is isomorphic to the colimit of such a system in which all
maps are injective with closed image if and only if it is isomorphic
to a countable direct sum of spaces with countable bases.

(c) If the $A_i$ are in addition linearly compact, then $A$ is
isomorphic to a countable direct sum of linearly compact spaces with
countable bases. In particular, $A=\colim A_i$ is complete.   
\end{lemma}

\begin{proof}
(a) Define the direct system $B_i:=f_i(A_i)\subset A$ with the
inclusions $B_i\into B_j$. Then $A_i\ni a_i\mapsto f_i(a_i)\in
B_i$ defines a map of direct systems which induces an isomorphism
$A\to\colim B$. 

(b) A direct sum $A=\oplus_{i\in\N}A_i$ is isomorphic to $\colim B_i$
for the system $B_i=A_1\oplus\cdots\oplus A_i$ with the inclusions 
$B_i\into B_j$ for $i<j$. Conversely, suppose that the maps
$f_{ji}:A_i\to A_j$ are injective with closed image. Since the $A_i$
have countable bases, by Corollary~\ref{cor:splitting-Positselski} we
find topological complements $B_i$ to $f_{i,i-1}(A_{i-1})$ such that (with $B_1=A_1)$
$$
  A_i = f_{i,i-1}(A_{i-1})\oplus B_i = f_{i1}(B_1)\oplus
  f_{i2}(B_2)\oplus\cdots\oplus f_{i,i-1}(B_{i-1})\oplus B_i.
$$
Then the map $B_i\ni b_i'\mapsto f_i(b_i)$ defines an isomorphism
$\oplus B_i\to A$.  

(c) If the $A_i$ are linearly compact, then so are the
$B_i=f_i(A_i)$ in the proof of (a), so the inclusions $B_i\into B_j$
have closed image and $A$ is isomorphic to a countable direct sum of
spaces with countable bases by (b). By the proof of (b), the summands
are closed subspaces of linearly compact spaces and therefore linearly compact.
Completeness follows from~\cite[Lemma 1.2]{Positselski}.  
\end{proof}

By dualizing Lemma~\ref{lem:colim_countable}(c) we obtain 

\begin{corollary} \label{cor:lim_countable}
Let $A=\lim A_i$ for a countable inverse system of discrete vector
spaces of countable dimension. Then $A$ is isomorphic to a countable
direct product of discrete spaces of countable dimension.  \qed 
\end{corollary}

\subsection{Limits and colimits of Tate vector spaces} \label{sec:lim-colim-Tate}

In this subsection we investigate under which conditions limits and
colimits of Tate vector spaces are Tate. 
We begin with a simple lemma. 

\begin{lemma}\label{lem:colim-lim-Tate}
(a) If $V_i$ is a direct system of discrete vector spaces,
  then $\colim V_i$ is discrete, hence Tate. 

(b) If $W_i$ is an inverse system of linearly compact vector spaces,
  then $\lim W_i$ is linearly compact, hence Tate.
\end{lemma}

\begin{proof}
(a) The preimage of $\{0\}$ under each map $V_i\to \colim V_i$ 
  is open, hence $\{0\}$ is open and $\colim V_i$ is discrete. 

(b) This is~\cite[(II.27.6) and (II.27.3)]{Lefschetz-book}. 
\end{proof}

The following example shows that countably indexed limits of discrete spaces and
colimits of linearly compact spaces are in general not Tate. 

\begin{example}\label{example:limits-colimits-of-Tate}
(a) Let $(W_i,g_{ij})$ be an inverse system of discrete spaces indexed
over $\N$ with $g_{ij}$ surjective and $\ker g_{ij}$ infinite
dimensional for all $j>i$. Then $\lim_{i\in\N}W_i$ is not Tate. 

(b) Let $V_i=\bk[[t_1,\dots,t_i]]$ with the inclusions
$f_{ij}:V_i\into V_j$, $i\le j$. Then each $V_i$ is linearly compact,
but $\colim_{i\in\N}V_i$ is not Tate.
\end{example}

Recall that a continuous linear map $f:A\to B$ between linearly topologized spaces is {\em open} if images of open linear subspaces are open. We call $f$ {\em proper} if preimages of linearly compact subspaces are linearly compact. These notions are dual: if $f$ is open then $f^*$ is proper, and vice versa. Note that if $f$ is open and $A$ is Tate then $B$ is Tate, and if $f$ is proper and $B$ is Tate then $A$ is Tate.

The next proposition is the main result of this subsection. 

\begin{proposition}\label{prop:limit-colimit-Tate}
(a) If $(g_{ij}:W_j\to W_i)$ is an inverse system of Tate vector
spaces indexed over $\N$ and $g_{ij}$ is proper
for each $i\leq j$, then $\lim W_i$ is Tate.

(b) If $(f_{ji}:V_i\to V_j)$ is a direct system of Tate vector
spaces indexed over $\N$ and $f_{ji}$ is open
for each $i\leq j$, then $\colim V_i$ is Tate.
\end{proposition}

\begin{proof}
(a) Pick an open linearly compact subspace $U_1\subset W_1$. By hypothesis, the subspaces $U_j:=g_{1j}^{-1}(U_1)\subset W_j$ are open and linearly compact for each $j\geq 2$. Thus for all $i<j$ we get a commutative diagram of short exact sequences 
$$
\xymatrix{
  0 \ar[r] & U_i \ar[r] & W_i \ar[r] & W_i/U_i \ar[r] & 0 \\
  0 \ar[r] & U_j \ar[r] \ar@{->>}[u]^-{\tilde g_{ij}} & W_j  \ar[u]_-{g_{ij}} \ar[r] & W_j/U_j \ar@{^(->}[u]_-{\bar g_{ij}} \ar[r] & 0
}
$$
where $\tilde g_{ij}$ is surjective and $\bar g_{ij}$ is injective, and the spaces $W_i/U_i$ are discrete. For $i=1$ choose a linear splitting map $s_1:W_1/U_1\to W_1$, which is continuous because $W_1/U_1$ is discrete. Its image $D_1$ is a complement of $U_1$ in $W_1$, and it is closed because $U_1$ is open. According to Definition~\ref{defi:equivalence-splitting}, we get a topological splitting $W_1=U_1\oplus D_1$, and in particular a splitting map $\pi_1:W_1\to U_1$. Applying Lemma~\ref{lem:diagram} to the diagram for $i=1$, $j=2$ and reasoning as above, we find a topological splitting $W_2=U_2\oplus D_2$ such that $g_{12}$ sends $U_2$ to $U_1$ and $D_2$ to $D_1$. Continuing inductively, we find topological splittings $W_i=U_i\oplus D_i$ such that $g_{ij}$ sends $U_j$ to $U_i$ and $D_j$ to $D_i$ for all $i<j$. It follows that
$$
  \lim W_i = \lim U_i\oplus \lim D_i.
$$
Now $\lim U_i$ is linearly compact because the $U_i$ are linearly compact. For the second summand, recall that the $D_i$ are discrete and the restrictions $h_{ij}=g_{ij}|_{D_j}:D_j\to D_i$ are injective by construction. Thus for the induced map $h_i:\lim D_j\to D_i$ the set $h_i^{-1}(0)=\{0\}$ is open, so $\lim D_i$ is discrete. Hence, $\lim W_i$ is Tate. 

(b) Pick an open linearly compact subspace $U_1\subset V_1$. By hypothesis, the subspaces $U_j:=f_{j1}(U_1)\subset V_j$ are open and linearly compact for each $j\geq 2$. Thus for all $i<j$ we get a commutative diagram of short exact sequences 
$$
\xymatrix{
  0 \ar[r] & U_j \ar[r] & V_j \ar[r] & V_j/U_j \ar[r] & 0 \\
  0 \ar[r] & U_i \ar[r] \ar@{->>}[u]^-{\tilde f_{ji}} & V_i  \ar[u]_-{f_{ji}} \ar[r] & V_i/U_i \ar@{^(->}[u]_-{\bar f_{ji}} \ar[r] & 0
}
$$
where $\tilde f_{ij}$ is surjective and the spaces $V_i/U_i$ are discrete. As in (a), we find topological splittings $V_i=U_i\oplus D_i$ such that $f_{ji}$ sends $U_i$ to $U_j$ and $D_i$ to $D_j$ for all $i<j$. It follows that
$$
  \colim V_i = \colim U_i\oplus \colim D_i.
$$
Now $\colim D_i$ is discrete because the $D_i$ are discrete. For the first summand, recall that the $U_i$ are linearly compact and the restrictions $\tilde f_{ji}=f_{ji}|_{U_i}:U_i\to U_j$ are surjective by construction. Thus for the induced map $\tilde f_i:U_i\to \colim U_j$ the set $\tilde f_i(U_i)=\colim U_j$ is linearly compact. Hence, $\colim V_i$ is Tate. 
\end{proof}

We will use in the sequel the following special case of
Proposition~\ref{prop:limit-colimit-Tate}. 

\begin{corollary}\label{cor:limit-colimit-Tate}
(a) If $(g_{ij}:W_j\to W_i)$ is an inverse system of discrete vector
spaces indexed over $\N$ and $\ker g_{ij}$ is finite dimensional
for each $i\leq j$, then $\lim W_i$ is Tate.

(b) If $(f_{ji}:V_i\to V_j)$ is a direct system of linearly compact vector
spaces indexed over $\N$ and $\coker f_{ji}$ is finite dimensional
for each $i\leq j$, then $\colim V_i$ is Tate.
\end{corollary}

\begin{proof}
We only need to verify the hypotheses of Proposition~\ref{prop:limit-colimit-Tate}.

(a) If $W_i$ is discrete and $\ker g_{ij}$ is finite dimensional, then
$g_{ij}:W_j\to W_i$ is proper.

(b) If $V_i$ is linearly compact and $\coker f_{ji}$ is finite
dimensional, then each open linear subspace $U\subset V_i$ is linearly
compact and $f_{ji}(U)\subset V_j$ is a closed linear subspace of
finite codimension, hence open by Lemma~\ref{lem:open}. 
\end{proof}

\subsection{Limits, colimits, and duality of Tate vector spaces} \label{sec:lim-colim-duality}

In this subsection we revisit Proposition~\ref{prop:limit-colimit-Tate} 
in relation with duality.

\begin{proposition}\label{prop:lim-colim-dual}
(a) Consider a direct system $(V_i,f_{ji})$ of Tate vector 
spaces indexed over $\N$ such that $f_{ji}$ is
open for each $i\leq j$. Then the canonical map 
$$
   f: (\colim V_i)^*\to \lim V_i^* 
$$
is an isomorphism of Tate vector spaces. 

(b) Consider an inverse system $(W_i,g_{ij})$ of Tate vector 
spaces indexed over $\N$ such that $g_{ij}$ is proper
for each $i\leq j$. Then the canonical map 
$$
   g: \colim W_i^*\to (\lim W_i)^*
$$
is an isomorphism of Tate vector spaces. 
\end{proposition}

\begin{proof}
(a) The space $\colim V_i$ is Tate by
Proposition~\ref{prop:limit-colimit-Tate}(b), so its dual is Tate by 
Theorem~\ref{thm:Tate-duality}(c). The duals $f_{ji}^*:V_j^*\to V_i^*$
are proper maps between Tate vector spaces, so $\lim V_i^*$ is Tate by
Proposition~\ref{prop:limit-colimit-Tate}(a).
The isomorphism follows from Corollary~\ref{cor:lim-colim-duality}(a).

(b) The space $\lim W_i$ is Tate by
Proposition~\ref{prop:limit-colimit-Tate}(a), so its dual is Tate by 
Theorem~\ref{thm:Tate-duality}(c). The duals $g_{ij}^*:W_i^*\to W_j^*$
are open maps between Tate vector spaces, so
$\colim W_i^*$ is Tate by Proposition~\ref{prop:limit-colimit-Tate}(b).
The isomorphism follows from Corollary~\ref{cor:lim-colim-duality}(b)
(alternatively, it can be deduced from part (a) via
Theorem~\ref{thm:Tate-duality}(c)). 
\end{proof}

We will use in the sequel the following special case of
Proposition~\ref{prop:lim-colim-dual},
which is deduced as in the proof of Corollary~\ref{cor:limit-colimit-Tate}.

\begin{corollary}\label{cor:lim-colim-dual}
(a) Consider a direct system $(V_i,f_{ji})$ of linearly compact vector 
spaces indexed over $\N$ such that $\coker f_{ji}$ is
finite dimensional for each $i\leq j$. Then the canonical map 
$$
   f: \lim V_i^* \to (\colim V_i)^*
$$
is an isomorphism of Tate vector spaces. 

(b) Consider an inverse system $(W_i,g_{ij})$ of discrete vector 
spaces indexed over $\N$ such that $\ker g_{ij}$ is
finite dimensional for each $i\leq j$. Then the canonical map 
$$
   g: \colim W_i^*\to (\lim W_i)^*
$$
is an isomorphism of Tate vector spaces. 
\qed
\end{corollary}

\section{Rabinowitz Floer homology is Tate$^\omega$}\label{sec:RFH}

In this section we apply the previous sections to Rabinowitz Floer
homology and prove Theorems~\ref{thm:main1} and \ref{thm:main2} from
the Introduction (except for the statement on the Poincaré duality isomorphism from Theorem~\ref{thm:main1}, which we prove as part of Theorem~\ref{thm:RFHmulambdaPD} in the next section). 

Rabinowitz Floer homology $RFH_*(\p V)$ is a homological invariant associated to the boundary of a Liouville domain $V$. It is defined via its action truncated versions $RFH_*^{(a,b)}(\p V)$, $-\infty<a<b<\infty$ as 
$$
RFH_*(\p V) = \colim_b \lim_a \, RFH_*^{(a,b)}(\p V),\qquad a\to-\infty, \qquad b\to\infty.
$$
Here the collection $RFH_*^{(a,b)}=RFH_*^{(a,b)}(\p V)$ fits into a \emph{bidirected system}, meaning that we have commutative diagrams 
$$
\xymatrix{
RFH_*^{(a_1,b_1)}\ar[d] \ar[r] & RFH_*^{(a_2,b_1)} \ar[d] \\
RFH_*^{(a_1,b_2)} \ar[r] & RFH_*^{(a_2,b_2)}
}
$$ 
for $a_1\le a_2$ and $b_1\le b_2$. We interpret this bidirected system as being an inverse system in $a\to-\infty$ and a direct system in $b\to\infty$. 

Rabinowitz Floer cohomology $RFH^*_{(a,b)}$ is also defined via its action truncated versions $RFH^*_{(a,b)}(\p V)$, $-\infty<a<b<\infty$ as 
$$
RFH^*(\p V) = \colim_a \lim_b \, RFH^*_{(a,b)}(\p V),\qquad a\to-\infty, \qquad b\to\infty.
$$
The collection $RFH^*_{(a,b)}=RFH^*_{(a,b)}(\p V)$ also fits into a bidirected system, with commutative diagrams 
$$
\xymatrix{
RFH^*_{(a_1,b_1)}  & RFH^*_{(a_2,b_1)} \ar[l] \\
RFH^*_{(a_1,b_2)} \ar[u]  & RFH^*_{(a_2,b_2)} \ar[u] \ar[l]
}
$$ 
for $a_1\le a_2$ and $b_1\le b_2$. We interpret this bidirected system as being a direct system in $a\to-\infty$ and an inverse system in $b\to\infty$.

The action truncated Rabinowitz Floer (co)homology groups $RFH_*^{(a,b)}$ and $RFH^*_{(a,b)}$ are finite dimensional and thus naturally linearly topologized as discrete vector spaces. As a consequence, both $RFH_*(\p V)$ and $RFH^*(\p V)$ carry natural linear topologies.

Our main result in this section is the following

\begin{theorem} \label{thm:RFH-self-dual}
Rabinowitz Floer homology and cohomology of any Liouville domain $V$ are Tate$^\omega$ vector spaces such that $RFH^*(\p V)$ is the topological dual of $RFH_*(\p V)$. 
\end{theorem}

We recall from~\S\ref{sec:countability_hypotheses} that Tate$^\omega$ vector spaces are Tate spaces that have countable basis and possess a c-lattice of countable codimension. 

The following theorem provides the key technical ingredient for the proof and is a strengthening of Theorem~\ref{thm:main2} from the Introduction. It contains in particular a commutation result for $\lim$ and $\colim$.\footnote{In general, limits and colimits do not commute for bidirected systems, see e.g.~\cite{CF}.} 

\begin{theorem}[Commuting $\lim$ and $\colim$]\label{thm:kappa-top} \qquad 

(a) The source and target of the canonical linear continuous map
$$
  \kappa\colon 
  \colim_b \lim_a \, RFH_*^{(a,b)}
  \to \lim_a \colim_b \, RFH_*^{(a,b)}
$$
are Tate$^\omega$ vector spaces, and $\kappa$ is a topological isomorphism.

(b) The source and target of the canonical linear continuous map
$$
  \kappa\colon 
  \colim_a \lim_b \, RFH^*_{(a,b)}
  \to \lim_b \colim_a \, RFH^*_{(a,b)}
$$
are Tate$^\omega$ vector spaces, and $\kappa$ is a topological isomorphism.
\end{theorem}

In this statement, the relevant vector spaces carry natural linear topologies induced by the discrete linear topologies on the action truncated Rabinowitz Floer (co)homology groups, which are finite dimensional.

\begin{proof}[Proof of Theorem~\ref{thm:RFH-self-dual} assuming Theorem~\ref{thm:kappa-top}] \qquad 

That $RFH_*(\p V)$ and $RFH^*(\p V)$ are Tate$^\omega$ vector spaces is the content of Theorem~\ref{thm:kappa-top}. 
The duality statement follows from the canonical isomorphisms
\begin{align*}
  (RFH_*(\p V))^* &= (\colim_b\lim_a RFH_*^{(a,b)})^* \\
  &\cong \lim_b(\lim_a RFH_*^{(a,b)})^* \\
  &\cong \lim_b\colim_a (RFH_*^{(a,b)})^* \\
  &\cong \lim_b\colim_a RFH^*_{(a,b)} \\
  &\cong \colim_a\lim_b RFH^*_{(a,b)} \\
  &= RFH^*(\p V).
\end{align*}
Here the equalities in the first and last line are the
definitions.
The first isomorphism follows from Corollary~\ref{cor:lim-colim-duality}(a)
(where $\lim_a RFH_*^{(a,b)}$ is linearly compact (hence Tate) as a limit of finite dimensional (hence linearly compact) vector spaces); the second one from Corollary~\ref{cor:lim-colim-duality}(b); the third one from the universal coefficient theorem and finite
dimensionality of $RFH^*_{(a,b)}$; and the fourth one from Theorem~\ref{thm:kappa-top}(b). 
\end{proof}

\subsection{Basis of open neighborhoods of $0$} \label{sec:basis-of-opens}

The description of Rabinowitz Floer homology as the \emph{target of $\kappa$} from Theorem~\ref{thm:kappa-top}(a) allows an easy and intuitive description of a basis of open neighborhoods of $0$. Let 
$$
RFH_*^{(c,\infty)}=\colim_b RFH_*^{(c,b)}
$$
and consider the canonical map 
$$
\pi_c:RFH_*(\p V)\to RFH_*^{(c,\infty)}
$$
stemming from $\lim_c\colim_b RFH_*^{(c,b)} \to \colim_b RFH_*^{(c,b)}$. A basis of open neighborhoods of $0$ for $RFH_*(\p V)$ is given by 
\begin{equation} \label{eq:Uc}
U_c = \ker\left(\pi_c:RFH_*(\p V)\to RFH_*^{(c,\infty)}\right), \qquad c\in\R.
\end{equation}
Note that, by general considerations, a basis of $RFH_*(\p V)$ is given by the open subspaces of the form $\pi_c^{-1}(O_c)$, where $c\in\R$ and $O_c\subset RFH_*^{(c,\infty)}$ belongs to a basis of neighborhoods of $0$ in $RFH_*^{(c,\infty)}$. Since this space is discrete, we can take such a basis to consist only of $\{0\}$ and therefore only consider the open subspaces $U_c=\pi_c^{-1}(0)$.

The description of $RFH_*(\p V)$ as the \emph{source of $\kappa$} from Theorem~\ref{thm:kappa-top}(a) yields an alternative description of this basis of open sets in terms of the linearly compact vector spaces 
$$
RFH_*^{(-\infty,c)} = \lim_a RFH_*^{(a,c)}
$$
and of the canonical maps 
$$
\iota_c:RFH_*^{(-\infty,c)}\to RFH_*(\p V)
$$
stemming from $\lim_a RFH_*^{(a,c)}\to \colim_c\lim_a RFH_*^{(a,c)}$. The system of canonical long exact sequences of finite dimensional vector spaces $\dots \to RFH_*^{(a,c)} \to RFH_*^{(a,b)}\to RFH_*^{(c,b)}\to\dots$, defined for $a<c<b$, yields by first passing to the limit over $a\to-\infty$ and then to the colimit as $b\to\infty$ a long exact sequence 
$$
\dots\to RFH_*^{(-\infty,c)} \stackrel{\iota_c}\to RFH_*(\p V)\stackrel{\pi_c}\to RFH_*^{(c,\infty)}\to\dots
$$
(The identification of the map $RFH_*(\p V)\to RFH_*^{(c,\infty)}$ with $\pi_c$ is straightforward from the definition of $\kappa$.) Thus 
$$
U_c = \ker \pi_c = \im \iota_c.
$$

\begin{remark}
There is of course a similar description of a basis of linear open neighborhoods of $0$ for Rabinowitz Floer cohomology, relying on Theorem~\ref{thm:kappa-top}(b). We leave the details to the interested reader. 
\end{remark}

\subsection{Commuting $\lim$ and $\colim$ for Rabinowitz Floer homology}

Our proof of Theorem~\ref{thm:kappa-top} relies on the following more abstract statement. 

\begin{proposition} \label{prop:kappa-abstract} Let $V_{i,j}$ be a bidirected system of finite dimensional vector spaces indexed by $i\in\N^-$, $j\in \N^+$. Assume that there exist an inverse system $W_i$, $i\in \N^-$ and a direct system $V_j$, $j\in\N^+$ with a short exact sequence of bidirected systems\footnote{Here we interpret $V_j$ and $W_i$ as bidirected systems constant in $i$, resp. $j$.}   
$$
0\to V_j\to V_{i,j}\to W_i\to 0.
$$ 
Then the source and target of the canonical map
$$
\kappa:\colim_j\lim_i V_{i,j}\to \lim_i  \colim_j V_{i,j}
$$ 
are Tate$^\omega$ vector spaces, and $\kappa$ is a topological isomorphism. 
\end{proposition}

\begin{proof} The existence of the map $\kappa$ is a straightforward consequence of the universal properties of limits and colimits: for fixed $j$ we have a map of inverse systems $V_{i,j}\to \colim_j V_{i,j}$, hence a map $\lim_i V_{i,j}\to \lim_i \colim_j V_{i,j}$ with source the direct system $\lim_i V_{i,j}$, wherefrom the map $\kappa: \colim_j \lim_i V_{i,j} \to \lim_i \colim_j V_{i,j}$. 

Since the $V_{i,j}$ are finite dimensional, so are the $V_j$ and the $W_i$. 

That $\colim_j\lim_i V_{i,j}$ is Tate is a consequence of Corollary~\ref{cor:limit-colimit-Tate}(b). To this effect, we use that $\lim$ is exact on inverse systems that consist of finite dimensional vector spaces (see~\cite[Lemma~3.6]{CF} and~\cite[Proposition~3.5.7]{Weibel}). We infer a direct system of short exact sequences $0\to V_j \to \lim_i V_{i,j}\to \lim_iW_i\to 0$, hence isomorphisms $\coker(\lim_i V_{i,j}\to \lim_i V_{i,j+1})\simeq \coker(V_j\to V_{j+1})$. The latter space is finite dimensional, so Corollary~\ref{cor:limit-colimit-Tate}(b) applies to show that $\colim_j \lim_i V_{i,j}$ is Tate.  

That $\lim_i\colim_j V_{i,j}$ is Tate is a consequence of Corollary~\ref{cor:limit-colimit-Tate}(a). Since $\colim$ is exact we infer an inverse system of short exact sequences $0\to \colim_j V_j\to \colim_j V_{i,j}\to W_i\to 0$, so that we have isomorphisms $\ker (\colim_j V_{i-1,j}\to \colim_j V_{i,j})\simeq \ker (W_{i-1}\to W_i)$. The latter space is finite dimensional, so Corollary~\ref{cor:limit-colimit-Tate}(a) applies to show that $\lim_i \colim_j  V_{i,j}$ is Tate. 

We claim that there is an isomorphism of bidirected systems $V_{i,j}\simeq V_j\oplus W_i$. This implies the conclusion of the proposition as follows. We have $\colim_j\lim_i (V_j\oplus W_i)\simeq \colim_j V_j\oplus \lim_i W_i \simeq \lim_i \colim_j (V_j\oplus W_i)$, and the composition of these isomorphisms is $\kappa$. Note that the middle term is Tate because $\colim_j V_j$ is discrete as a colimit of discrete spaces, and $\lim_i W_i$ is linearly compact as a limit of linearly compact spaces. This gives an alternative proof of the fact that $\colim_j\lim_i V_{i,j}$ and $\lim_i \colim_j  V_{i,j}$ are Tate. 

The space $\colim_j V_j\oplus \lim_i W_i$ is actually Tate$^\omega$. Indeed $\colim_j V_j$ is discrete of countable dimension as a countable colimit of finite dimensional spaces, and $\lim_i W_i$ is linearly compact with countable basis as a countable limit of finite dimensional spaces. Moreover $\lim_i W_i$ is open, and hence a c-lattice, because $\colim_j V_j$ is discrete. This proves that $\colim_j V_j\oplus \lim_i W_i$ is Tate $^\omega$, and the same holds for the source and target of $\kappa$. 

We now prove the claim. Let $g_{i,i'}:W_{i'}\to W_i$, $i'\le i$, $i,i'\in\N^-$ be the maps in the inverse system $(W_i)$, and let $f_{j',j}:V_j\to V_{j'}$, $j\le j'$, $j,j'\in \N^+$ be the maps in the direct system $(V_j)$. The space $V_j$ injects into $V_{i,j}$ and we identify it canonically with a subspace of $V_{i,j}$. We need to find a direct summand $E_{i,j}$ for $V_j$ in $V_{i,j}$ such that, upon writing $V_{i,j}=V_j\oplus E_{i,j}$ and identifying $E_{i,j}$ with $W_i$ via the projection $V_{i,j}\to W_i$, the maps $V_{i-1,j}\to V_{i,j}$ are identified with $1\oplus g_{i,i-1}$ and the maps $V_{i,j}\to V_{i,j+1}$ are identified with $f_{j+1,j}\oplus 1$. 

We construct the requested decompositions for $-n\le i\le -1$ and arbitrary $j$ by induction over $n$. 

Let $n=-1$. We construct by induction over $j$ decompositions $V_{-1,j}=V_j\oplus E_{-1,j}$ such that the maps $V_{-1,j}\to V_{-1,j+1}$ are identified with $f_{j+1,j}\oplus 1$. We start with an arbitrary direct sum decomposition $V_{-1,1}=V_1\oplus E_{-1,1}$. Having chosen the direct summand  $E_{-1,j}$ of $V_j$ in $V_{-1,j}$, let $E'_{-1,j+1}$ be some direct summand for $V_{j+1}$ in $V_{-1,j+1}$. The map $V_{-1,j}\to V_{-1,j+1}$ is then identified with $\begin{pmatrix} f_{j+1,j} & \tau_{j+1,j}\\ 0 & 1\end{pmatrix}$ for some $\tau_{j+1,j}:E_{-1,j}\to V_{j+1}$, and we choose $E_{-1,j+1}=\mathrm{graph}(-\tau_{j+1,j})$. This completes the induction step. 

Assume we have constructed such decompositions for $-n\le i\le -1$ and arbitrary $j$. We construct them for $i=-n-1$ and arbitrary $j$ as follows. For each $j$ let $E'_{-n-1,j}$ be some direct summand of $V_j$ in $V_{-n-1,j}$. The maps $V_{-n-1,j}\to V_{-n,j}$ are identified with $\begin{pmatrix} 1 & \sigma_{j,-n-1}\\ 0 & g_{-n,-n-1}\end{pmatrix}$ for some $\sigma_{j,-n-1}:E'_{-n-1,j}\to V_j$. Let $E_{-n-1,j}=\mathrm{graph}(-\sigma_{j,-n-1})$, so that the map $V_{-n-1,j}\to V_{-n,j}$ is identified with $1\oplus g_{-n,-n-1}$. Consider the commutative square 
$$
\xymatrix{
V_{-n,j} \ar[r] & V_{-n,j+1}\\
V_{-n-1,j}\ar[u] \ar[r] & V_{-n-1,j+1} \ . \ar[u]
} 
$$
The composition of the left vertical and upper horizontal maps is $f_{j+1,j}\oplus g_{-n,-n-1}$. The right vertical map is $1\oplus g_{-n,-n-1}$, and the bottom horizontal map is upper triangular of the form $\begin{pmatrix} f_{j+1,j} & a \\ 0 & 1\end{pmatrix}$ for some $a:E_{-n-1,j}\to V_{j+1}$. Commutativity of the square implies $a=0$ and the bottom horizontal map is therefore $f_{j+1,j}\oplus 1$. This completes the induction step, and therefore the proof of the claim, and therefore the proof of the proposition.
\end{proof}

To prove Theorem~\ref{thm:kappa-top} we rely on the following finite action enhancement of the decomposition~\eqref{eq:RFH-SH-bar}. For $a<b$ consider the symplectic homology $SH_*^{(a,b)}$,  
symplectic cohomology $SH^*_{(a,b)}$, and Rabinowitz
Floer homology $RFH_*^{(a,b)}$ of the Liouville domain $V$ in the action interval $(a,b)$.  
According to~\cite[Proposition 2.9]{Cieliebak-Frauenfelder-Oancea},
for all $a<0<b$ we have a long exact sequence of finite dimensional vector spaces
\begin{equation*}
\xymatrix
@C=15pt
{
\ldots \ar[r] &  \!SH^{-*}_{(-\infty,-a)}\! \ar[r]^-\eps & \!SH_*^{(-\infty,b)}\! \ar[r] &
\!RFH_*^{(a,b)}\! \ar[r]
& \!SH^{1-*}_{(-\infty,-a)}\! \ar[r]^-\eps & \ldots
}
\end{equation*}
Passing to reduced symplectic homology $\ol{SH}_*=\coker\eps$ and cohomology
$\ol{SH}^*=\ker\eps$, this becomes a short exact sequence
\begin{equation}\label{eq:SES-ab}
\xymatrix
@C=15pt
{
  0 \ar[r] & \!\ol{SH}_*^{(-\infty,b)}\! \ar[r] & \!RFH_*^{(a,b)}\! \ar[r]
  & \!\ol{SH}^{1-*}_{(-\infty,-a)}\! \ar[r] & 0\,.
}
\end{equation}
The groups $RFH_*^{(a,b)}$ for varying $a,b$ define a bidirected system,
the groups $\ol{SH}_*^{(-\infty,b)}$ and
$\ol{SH}^{1-*}_{(-\infty,-a)}$ define a direct and an inverse system respectively, and~\eqref{eq:SES-ab} defines a short exact sequence of bidirected systems. 

Similar considerations apply to the cohomology groups $RFH^*_{(a,b)}$.

\begin{proof}[Proof of Theorem~\ref{thm:kappa-top}]
We give the details only for (a), since the proof of (b) is entirely analogous. Consider the short exact sequence~\eqref{eq:SES-ab}, 
\begin{equation*}
\xymatrix
@C=15pt
{
  0 \ar[r] & \!\ol{SH}_*^{(-\infty,b)}\! \ar[r] & \!RFH_*^{(a,b)}\! \ar[r]
  & \!\ol{SH}^{1-*}_{(-\infty,-a)}\! \ar[r] & 0\,,
}
\end{equation*}
defined for $a<0<b$.  
These groups, and also $RFH_*^{(a,b)}$, can only possibly change when $|a|, b>0$ belong to the action spectrum of $\p V$, which is discrete, so that we can assume without loss of generality that these (bi)directed systems are of the form $V_j=\ol{SH}_*^{(-\infty,b_j)}$, $W_i=\ol{SH}^{1-*}_{(-\infty,-a_i)}$, $V_{i,j}=RFH_*^{(a_i,b_j)}$ for $i\in \N^-$, $j\in \N^+$ and $|a_i|, b_j$ chosen appropriately outside of the spectrum of $\p V$. Each of these groups is finite dimensional because the underlying Floer chain complexes are finite dimensional, so Proposition~\ref{prop:kappa-abstract} applies to show that $\kappa$ is a topological isomorphism of Tate$^\omega$ vector spaces. 
\end{proof}

\begin{remark} Passing in~\eqref{eq:SES-ab} first to the limit over $a\to -\infty$ and then to the colimit over $b\to\infty$ we obtain the short exact sequence 
\begin{equation*}
\xymatrix
@C=15pt
{
  0 \ar[r] & \!\ol{SH}_*\! \ar[r] & \!RFH_*\! \ar[r]
  & \!\ol{SH}^{1-*}\! \ar[r] & 0\,.
}
\end{equation*}
This short exact sequence splits in the sense of Definition~\ref{defi:equivalence-splitting}, resulting in a topological decomposition
\begin{equation*}
   RFH_* = \ol{SH}_* \oplus \ol{SH}^{1-*}.
\end{equation*}
See also~\eqref{eq:RFH-SH-bar}. The splitting $\ol{SH}^{1-*}\into RFH_*$ can be chosen to
contain the open subspace $U_{-\eps}$ for $\eps>0$ from~\eqref{eq:Uc}, so that 
$\ol{SH}^{1-*}\subset RFH_*$ is an open complement of the d-lattice
$\ol{SH}_*$, and therefore a c-lattice. 
\end{remark}

\section{Product and coproduct on Rabinowitz Floer homology} \label{sec:RFHprodcoprod}

In this section we prove Theorem~\ref{thm:main3} and the remaining statement from Theorem~\ref{thm:main1}. For readability we split these statements in two. The first part is 

\begin{theorem}\label{thm:RFHmulambdacont}
For any Liouville domain $V$, the pair-of-pants product $\mu$ and the
secondary pair-of-pants coproduct $\lambda$ on Rabinowitz Floer
homology $RFH_*=RFH_*(\p V)$ define continuous linear maps
\begin{gather*}
  \mu: RFH_*\hatotimes^*RFH_*\to RFH_* \quad\text{and}\quad
  \lambda: RFH_*\to RFH_*\hatotimes^!RFH_*. 
\end{gather*}
\end{theorem}

The key takeaway from the following proof is that continuity of the product is established by presenting $RFH_*$ as $\colim_b\lim_aRFH_*^{(a,b)}$, whereas continuity of the coproduct is established by presenting $RFH_*$ as $\lim_a\colim_bRFH_*^{(a,b)}$. That the two presentations are equivalent is the content of the previous Theorem~\ref{thm:kappa-top}(a). 

\begin{proof}[Proof of Theorem~\ref{thm:RFHmulambdacont}]
We first treat the pair-of-pants product $\mu$. This is induced by the action-truncated pair-of-pants products, i.e., the family of linear maps 
$$
\mu^{(a,b)}:RFH_*^{(a,b)}\otimes RFH_*^{(a,b)}\to RFH_*^{(a+b,2b)} 
$$
indexed by $-\infty<a<b<\infty$. Specifically, we have 
$\mu=\colim_b\lim_a\mu^{(a,b)}$ (with $RFH_*=\colim_b\lim_aRFH_*^{(a,b)}$).  

The source and target of the action-truncated pair-of-pants product $\mu^{(a,b)}$ are finite dimensional and hence both discrete and linearly compact. Let $RFH_*^{<b}=\lim_aRFH_*^{(a,b)}$, $a\to-\infty$. Since the category $\cP$ of pro-discrete spaces is stable under the $\hatotimes^!$ tensor product (Proposition~\ref{prop:duality_intertwines}), and since the three flavors of completed tensor products coincide on linearly compact spaces (Proposition~\ref{lem:A*otimesBlinearlycompact}), we obtain a family of linear continuous maps 
$$
\mu^{< b}=\lim_a \mu^{(a,b)}: RFH_*^{<b}\hatotimes^* RFH_*^{<b} = RFH_*^{<b}\hatotimes^! RFH_*^{<b}\to RFH_*^{<2b}.
$$
Since the category $\cI$ of ind-linearly compact spaces is stable under the $\hatotimes^*$ tensor product (Proposition~\ref{prop:duality_intertwines}), we find that 
$$
\mu=\colim_b \mu^{<b}:RFH_*\hatotimes^*RFH_*\to RFH_*
$$
is linear continuous. 

We now treat the secondary pair-of-pants coproduct $\lambda$. For a given admissible Hamiltonian $H$ the action-truncated secondary pair-of-pants coproduct acts as
\begin{align*}
\lambda^{<b}: FH_*^{<b}(H) &\to FH_*^{<\frac b 2}(H)\otimes FH_*(H) + FH_*(H)\otimes FH_*^{<\frac b 2}(H)\\
& \to FH_*(H)\otimes FH_*(H),
\end{align*}
and further as 
\begin{align*}
& \lambda^{(a,b)}: FH_*^{(a,b)}(H) \to \frac{FH_*^{<\frac b 2}(H)\otimes FH_*(H) + FH_*(H)\otimes FH_*^{<\frac b 2}(H)}{FH_*^{<\frac a 2}(H)\otimes FH_*(H) + FH_*(H)\otimes FH_*^{<\frac a 2}(H)}\\
& \qquad\qquad = FH_*^{(\frac a 2,\frac b 2)}(H)\otimes FH_*^{>\frac a 2}(H) + FH_*^{>\frac a 2}(H)\otimes FH_*^{(\frac a 2,\frac b 2)}(H)\\
& \qquad\qquad  \to FH_*^{>\frac a 2}(H)\otimes FH_*^{>\frac a 2}(H). 
\end{align*}
The source and target of this map are finite dimensional, hence both discrete and linearly compact. Using the fact that the category $\cI$ of ind-linearly compact spaces is stable under the $\hatotimes^*$ tensor product (Proposition~\ref{prop:duality_intertwines}), and the fact that the $\hatotimes^*$ tensor product coincides with the algebraic tensor product on discrete spaces (Lemma~\ref{lem:tensor_discrete}), by taking the colimit over $H$ we find 
\begin{align*}
\lambda^{(a,b)}:RFH_*^{(a,b)}\to & RFH_*^{(\frac a 2,\frac b 2)}\otimes RFH_*^{>\frac a 2} + RFH_*^{>\frac a 2}\otimes RFH_*^{(\frac a 2,\frac b 2)}\\
& \qquad \qquad \subset RFH_*^{>\frac a 2}\otimes RFH_*^{>\frac a 2},
\end{align*}
and by further taking the colimit over $b\to\infty$ we find 
$$
\lambda^{>a}:RFH_*^{>a}\to RFH_*^{>\frac a 2}\otimes RFH_*^{>\frac a 2}.
$$ 
Here the source and target are discrete spaces, and the algebraic tensor product coincides with any of the completed tensor products. To conclude we take the limit over $a\to-\infty$ and, using that the category $\cP$ of pro-discrete vector spaces is stable under the $\hatotimes^!$ tensor product, we find that the pair-of-pants product 
$$
\lambda=\lim_a\lambda^{>a}:RFH_*\to RFH_*\hatotimes^! RFH_*
$$
is linear and continuous.
\end{proof}

The second part 
is

\begin{theorem}\label{thm:RFHmulambdaPD} For any Liouville domain $V$, the Poincar\'e duality isomorphism~\eqref{eq:PD} is a topological isomorphism that intertwines the pair-of-pants product $\mu$ and the
secondary pair-of-pants coproduct $\lambda$ on Rabinowitz Floer
homology $A=RFH_*(\p V)$, 
\begin{gather*}
  \mu: A\hatotimes^*A\to A \quad\text{and}\quad
  \lambda: A\to A\hatotimes^!A. 
\end{gather*}
with their topological duals on $A^*=RFH^*(\p V)$,  
\begin{gather*}
  \lambda^\vee: A^*\hatotimes^*A^* = (A\hatotimes^!A)^* \to A^*, \cr
  \mu^\vee: A^*\to (A\hatotimes^*A)^* = A^*\hatotimes^!A^*.
\end{gather*}
\end{theorem}

\begin{proof}
We only discuss the intertwining of $\mu$ and $\lambda^\vee$, since the case of $\lambda$ and $\mu^\vee$ is analogous. 

The proof is verbatim the same as that of the last assertion (c) in the proof of Theorem~4.8 in~\cite{CHO-PD}. We summarize it: Using the previous notation, the action-truncated versions of the Poincaré duality isomorphism $f_*$ and its inverse $g_*$ from~\cite{CHO-PD} fit into a commutative diagram  
$$
\xymatrix
@C=50pt
@R=20pt
{
   RFH^{(a,b)}_*\otimes RFH^{(a,b)}_* \ar[r]^{\ \ \ \ \ \ \mu^{(a,b)}} & RFH^{(a+b,2b)}_* \ar[d]^{f_*}_\cong \\
   RFH_{(-b,-a)}^{1-*}\otimes RFH_{(-b,-a)}^{1-*} \ar[u]^{g_*\otimes g_*}_\cong \ar[r]^{\ \ \ \ \ \ \lambda^\vee_{(-2b,-a-b)}} & RFH_{(-2b,-a-b)}^{1-*}.
}
$$
Here $\lambda^\vee_{(-2b,-a-b)}$ is the dual of the map $\lambda^{(a+b,2b)}$ from before.
\footnote{Note that here the choice of action intervals for the coproduct
differs from the one used in the proof of Theorem~\ref{thm:RFHmulambdacont}.}
The maps in the diagram are compatible with action filtrations, so taking the limit as $a\to-\infty$ and then the colimit as $b\to\infty$ we obtain a commutative diagram 
$$
\xymatrix
@C=30pt
@R=20pt
{
   RFH_*\otimes RFH_* \ar[r]^{\ \ \ \ \ \ \mu} 
   & RFH_* \ar[d]^{PD}_\cong \\
   RFH^{1-*}\otimes RFH^{1-*}  \ar[u]^{PD^{-1}\otimes PD^{-1}}_\cong  \ar[r]^{\ \ \ \ \ \ \ \ \ \lambda^\vee} & RFH^{1-*}.
}
$$
All the maps in the diagram are linear and continuous since they arise from action-truncated versions by taking $\colim/\lim$. In particular, $PD$ is a topological isomorphism, and it intertwines $\mu$ and $\lambda^\vee$. 
\end{proof}

\appendix

\section{Short exact sequences, topological complements, and Hahn-Banach} \label{sec:HBseq}

Throughout this appendix $A,B,C$ are linearly topologized vector
spaces over a discrete field $\bk$ and all maps are continuous linear.

\subsection{Short exact sequences}

\begin{definition}\label{def:ses}
A {\em short exact sequence} 
$$
0\to A \stackrel{i}\to B \stackrel{p}\to C\to 0
$$
is a short exact sequence of vector spaces such that $i$ is a closed
embedding and $p$ is an open surjection.
\end{definition}

Here $i$ being a closed embedding means that it is an isomorphism onto $i(A)$ which is closed; this is equivalent to $i$ being injective and continuous and mapping closed sets to closed sets. 
That $p$ is an open surjection is equivalent to it inducing an isomorphism from $B/A$ endowed with the quotient topology to $C$.

\begin{definition-lemma} \label{defi:equivalence-splitting}
A short exact sequence $0\to A \stackrel{i}\to B \stackrel{p}\to C\to
0$ {\em splits} if it satisfies one of the following equivalent conditions: 
\begin{enum}
\item there exists a surjection $\pi:B\to A$ such that $\pi i =\mathrm{Id}_A$. 
\item there exists an open surjection $\pi:B\to A$ such that $\pi i =\mathrm{Id}_A$. 
\item there exists a closed embedding $s:C\to B$ such that $ps=\mathrm{Id}_C$.
\item $B\simeq A\oplus C$ via $\pi$ and $s$ as in (ii) and (iii) such that $\pi s=0$. 
\end{enum}
$$
\xymatrix{
0\ar[r] &  A \ar[r]^i &  \ar@/^.5pc/[l]^-\pi B \ar[r]^-p &  \ar@/^.5pc/[l]^-s C \ar[r]  &  0.
}
$$
\end{definition-lemma}

\begin{proof}
(i) $\Rightarrow$ (ii). A continuous surjection $\pi:B\to A$ such that $\pi i =\mathrm{Id}_A$ is automatically open. To see this, note that $\pi|_{i(A)}:i(A)\to A$ is the inverse of $i:A\to i(A)$, hence open. Let $K=\ker \pi$, a closed complement of $i(A)$. Let $U\subset B$ be open. Then $U+K$ is open, and $\pi(U)=\pi(U+K)=\pi((U+K)\cap i(A))$ is open because $\pi|_{i(A)}$ is open.

(ii) $\Rightarrow$ (iii). Given $\pi$, the kernel $K=\ker\pi$ is a closed complement to $i(A)$ and $B=i(A)\oplus K$ is a topological direct sum. The isomorphism is realized by the maps $\xymatrix{\varphi:B\ar@<.5ex>[r] & \ar@<.5ex>[l] i(A)\oplus K:\psi}$ given by $\varphi(b)=(i\pi(b),b-i\pi(b))$ and $\psi(i(a),k)=i(a)+k$, which are continuous and inverse to each other. The map $p|_K:K\to C$ is an open continuous bijection, hence an isomorphism, and $s=(p|_K)^{-1}$ is a closed embedding. 

(ii) $\Rightarrow$ (iv). We have seen above that $B\simeq i(A)\oplus K$ and $p|_K:K\stackrel\simeq\longrightarrow C$. In view of $i:A\stackrel\simeq\longrightarrow i(A)$, we find that $B\simeq A\oplus C$. The image of the section $s:C\to B$ is by definition $K=\ker \pi$, so that $\pi s=0$. Specifically, the isomorphism $B\simeq A\oplus C$ is realized by the maps $\xymatrix{\varphi:B\ar@<.5ex>[r] & \ar@<.5ex>[l] A\oplus C:\psi}$ given by $\varphi(b)=(\pi(b),p(b))$ and $\psi(a,c)=i(a)+s(c)$, which are continuous and inverse to each other.

(iii) $\Rightarrow$ (ii). Assuming the existence of $s$, its image $K=s(C)$ is a closed complement of $i(A)$ and $B=i(A)\oplus K$ is a topological direct sum. The isomorphism is realized by the maps $\xymatrix{\varphi:B\ar@<.5ex>[r] & \ar@<.5ex>[l] i(A)\oplus K:\psi}$ given by $\varphi(b)=(b-sp(b),sp(b))$ and $\psi(i(a),k)=i(a)+k$, which are continuous and inverse to each other. 
Then the projection $\pi_1:B\to i(A)$ onto the first summand is an open surjection and $i:A\to i(A)$ is an isomorphism, so that $\pi=i^{-1}\pi_1:B\to A$ is an open surjection that satisfies $\pi i=\mathrm{Id}_A$ and $\pi s=0$. 

(iv) $\Rightarrow$ (ii), (iii). Obvious.
\end{proof}

Definition~\ref{def:ses} and Definition-Lemma~\ref{defi:equivalence-splitting} 
actually hold more generally for topological vector spaces. From now
on the restriction to linearly topologized spaces will be essential.

\begin{lemma} \label{lem:completeness-exact}
Let $0\to A \stackrel{i}\to B \stackrel{p}\to C\to 0$ be an exact sequence. If $B$ is complete, then $A$ and $C$ are complete.
\end{lemma}

\begin{proof}
The space $A$ is complete because it is isomorphic to $i(A)$, and a closed subspace of a complete space is complete. For $C$ we use that $p$ is an open surjection: let $\cW$ be a basis of open neighborhoods of $0$ for $C$, and $\cV$ a basis of open neighborhoods of $0$ for $B$. Then 
\begin{align*}
\lim_{W\in\cW} C/W & \simeq \lim_{W\in \cW}(B/A)/(p^{-1}(W)/A)\simeq \lim_{W\in\cW} B/p^{-1}(W)\\
& \simeq \lim_{V\in\cV} B/(V+A) \simeq \lim_{V\in\cV} (B/V)/(A/A\cap V) \simeq B/A\simeq C. 
\end{align*}  
Here the last isomorphism holds because $A$ and $B$ are complete. 
\end{proof}

\begin{proposition}[{Splitting~\cite[\S1.1, Remark]{Beilinson}}] \label{prop:splitting}
Let $0\to A \stackrel{i}\to B \stackrel{p}\to C\to 0$ be an exact sequence. Assume that $B$ is complete and has countable basis. Then the short exact sequence splits. 
\end{proposition}

\begin{proof}
We can assume w.l.o.g. that the basis of neighborhoods of $0$ in $B$ is nested, i.e., it has the form $V_1\supset V_2\supset\dots\supset V_k\supset \dots$. 

Let $V\subset U\subset B$ be open subspaces. We then have a commutative diagram
$$
\xymatrix{
0 \ar[r] & A/A\cap U \ar[r] & B/U \ar[r] & C/p(U) = B/(A+U)\ar[r] & 0 \\
0 \ar[r] & A/A\cap V \ar[r] \ar@{->>}[u] & B/V \ar[r] \ar@{->>}[u] & C/p(V) = B/(A+V)\ar[r] \ar@{->>}[u] & 0\,. 
}
$$ 
Since all the quotients are taken with respect to open subspaces, they are all discrete. In this context all linear maps are continuous. Given a splitting $\pi_U:B/U\to A/A\cap U$, there exists a splitting $\pi_V:B/V\to A/A\cap V$ that makes the first square in the above diagram commutative (Lemma~\ref{lem:diagram} below).

By induction, given our nested basis of open neighborhoods $V_k\supset V_{k+1}$ we can construct in this way a sequence of splittings $\pi_k:B/V_k\to A/A\cap V_k$ that define a map of inverse systems $\{\pi_k\}:\{B/V_k\}\to\{A/A\cap V_k\}$. The limit of the short exact sequences $0\to A/A\cap V_k\to B/V_k\to C/p(V_k)\to 0$ is exact because all the maps that define the inverse system $\{A/A\cap V_k\}$ are surjective~\cite[Lemma~3.5.3]{Weibel}, and $\pi=\lim \pi_k$, being a limit of splittings, is a splitting of the exact sequence  
$$
\xymatrix{
0\ar[r] & \lim_k A/A\cap V_k \ar[r] & \lim_k B/V_k \ar[r] \ar@/^/[l]^-{\pi} & \lim_k C/p(V_k) \ar[r] & 0\,.
}
$$   
Completeness of $B$ implies completeness of $A$ and $C$ by Lemma~\ref{lem:completeness-exact}. The map $\pi$ is therefore equivalent to a splitting of the initial exact sequence. 
\end{proof}

\begin{lemma} \label{lem:diagram}
Consider the following commutative diagram of short exact sequences of {\em
  discrete} vector spaces
$$
\xymatrix{
  0 \ar[r] & A_1 \ar[r]^-{i_1} & B_1  \ar@/^/[l]^-{\pi_1} \ar[r]^{p_1} & C_1 \ar@/^/@{.>}[l]^-{s_1} \ar[r] & 0 \\
0 \ar[r] & A_2 \ar[r]^-{i_2} \ar@{->>}[u]^-{f} & B_2  \ar[u]_-{g} \ar@/^/@{.>}[l]^-{\pi_2} \ar[r]^{p_2} & C_2 \ar[u]_-{h} \ar@/^/@{.>}[l]^-{s_2} \ar[r] & 0
}
$$
where $f$ is surjective and $\pi_1i_1=\Id_{A_1}$. Then there exist splitting maps $\pi_2,s_1,s_2$ that make the diagram commutative.
\end{lemma}

\begin{proof}
Since all maps between discrete spaces are continuous, this is just a
statement about vector spaces and linear maps. 
The space $S_1=\ker \pi_1$ is a complement of $\im i_1$ in $B_1$. Choose a complement $S_2$ of $\im i_2$ in $B_2$. Then the left square is equivalent to  
$$
\xymatrix{
A_1 \ar@{^(->}[r]^-{i_1} & A_1\oplus S_1  \ar@/^/[l]^-{\pi_1}  \\
A_2 \ar@{^(->}[r]^-{i_2} \ar@{->>}[u]^-{f} & A_2\oplus S_2  \ar[u]_-{g} \ar@/^/@{.>}[l]^-{\pi_2}
}
$$
where the maps $i_1,i_2$ are inclusions, the map $\pi_1$ is the projection on the first factor, and the map $g$ has upper triangular form $\begin{pmatrix} f & \alpha \\ 0 & \beta\end{pmatrix}$. We wish to construct a splitting $\pi_2$ so that the square commutes. Since $f$ is surjective, we can factor $S_2\stackrel{\alpha}{\longrightarrow} A_1$ as $S_2\stackrel{\theta}{\longrightarrow} A_2\stackrel{f}{\longrightarrow} A_1$. We replace $S_2$ by $S'_2=\mathrm{graph}(-\theta)$, which is still a complement of $A_2$ in $B_2$, so that the map $g$ has diagonal form $(f,\beta)$ with respect to the decomposition $A_2\oplus S'_2$. We define the splitting $\pi_2:A_2\oplus S'_2\to A_2$ to be the projection onto the first factor, and this makes the left square commutative.
Defining $s_i$ as the inverses of $p_i|_{\ker\pi_i}$ gives us splitting maps $s_i$ that make the right square commutative.
\end{proof}

\subsection{Topological complements} 

A {\em topological complement} of a\break closed subspace $A\subset B$ is a
closed subspace $S\subset B$ such that the canonical map $A\oplus S\to
B$ is a topological isomorphism. 

\begin{lemma}
A closed subspace $A\subset B$ has a topological complement if $A$ is
open or of finite codimension.
\end{lemma}

\begin{proof}
Assume that $A\subset B$ is open and pick an algebraic complement
$S\subset B$ of $A$. Then $S$ is discrete with the subspace topology
and one easily verifies that the canonical map $A\oplus S\to B$ is
continuous and open. If a closed subspace $A\subset B$ has finite
codimension, then it is open by Lemma~\ref{lem:open}.  
\end{proof}

\begin{example}
There exist examples of pairs $(B,A)$ with $B$ complete and $A\subset B$ closed such that $A$ does not admit a topological complement. To construct such examples we make the simple observation that, if $A\subset B$ admits a topological complement and $B$ is complete, then $B/A$ must be complete, being isomorphic to the complement of $A$, which is a closed subspace of $B$ and therefore complete. On the other hand, Positselski proves in~\cite[Theorem~3.1]{Positselski} that any linearly topologized vector space $Q$ can be realized as a quotient $B/A$ with $B$ complete.\footnote{Positselski attributes this result to Roelcke-Dierolf~\cite[Proposition~11.1]{Roelcke-Dierolf}, and further to K\"othe~\cite{Kothe}.} If $Q$ is not complete then $A\subset B$ does not admit a topological complement. 
\end{example}

Topological complements are closely related to short exact sequences:

\begin{lemma}
A closed subspace $A\subset B$ has a topological complement if and
only if the short exact sequence $0\to A\to B\to B/A\to 0$ splits.
\end{lemma}

\begin{proof}
Given a splitting $\pi:B\to A$ as in
Definition-Lemma~\ref{defi:equivalence-splitting}(i), the 
subspace $S=\ker \pi$ is a topological complement of $A$.  
Conversely, a topological complement $S$ satisfies $S\cong B/A$ and
gives therefore a splitting $B\simeq A\oplus B/A$.
\end{proof}

In view of this lemma, Proposition~\ref{prop:splitting} implies

\begin{corollary} \label{cor:exists-closed-complement}
Let $B$ be complete with countable basis. Then every closed
subspace $A\subset B$ has a topological complement.
\qed 
\end{corollary}

\begin{corollary}\label{cor:Hom-is-surjective}
Let $B$ be complete with countable basis. For any closed subspace $A\subset B$ and any $C$, the restriction map 
$$
\Hom(B,C)\to \Hom(A,C)
$$
is surjective.
\end{corollary}

\begin{proof}
By Corollary~\ref{cor:exists-closed-complement}, $A$ has a complement $S$ such that $B=A\oplus S$. Any map $A\to C$ can be extended by $0$ on $S$, and this is continuous because the direct sum decomposition $B=A\oplus S$ is topological. 
\end{proof}

\begin{example}\label{ex:non-open-mapping}
The Open Mapping Theorem and the Closed Graph Theorem, which are two foundational results from functional analysis on Banach spaces, fail in the current setting. Indeed, consider the following example: 

Let $V=\bk[[t]]$ with countable neighborhood basis of $0$ given by the
open sets $U_n=t^n\bk[[t]]$, $n\in\N$. Then $V$ is complete (and
linearly compact). Let $V_{\mathrm{discr}}$ be $\bk[[t]]$ with the discrete topology. This is again complete and has a countable neighborhood basis of $0$ (consisting of the single set $\{0\}$). We consider the identity map $\varphi=\mathrm{Id}:V_{\mathrm{discr}}\to V$. 
\begin{itemize}
\item This map is continuous and bijective, but not open because $\{0\}$ is not open in $V$. Therefore, the Open Mapping Theorem fails. 
\item The graph of $\varphi$ inside $V_{\mathrm{discr}}\oplus V$ is closed because $\varphi$ is continuous. The graph is the diagonal and it is also closed in $V\oplus V_{\mathrm{discr}}$, where it represents the graph of $\varphi^{-1}=\Id:B\to A$. However, the map $\varphi^{-1}$ is not continuous. Therefore the Closed Graph Theorem fails.
\end{itemize} 
\end{example}

\begin{example}
The same example shows that the following statement is false: ``Let $B$ be complete with countable basis. 
Given $A,S\subset B$ closed subspaces such that $A+S=B$ and $A\cap S=0$, the space $B$ equals the topological direct sum
$B = A\oplus S$."

Indeed, let $V=\bk[[t]]$ as above and $V_{\mathrm{discr}}$ the same vector space endowed with the discrete topology. Then $B=V_{\mathrm{discr}}\oplus V$ is complete and has countable basis. Let $\varphi=\Id:V_{\mathrm{discr}}\to V$. Consider the closed subspaces $S=\mathrm{graph}(\varphi)$ and $A=V_{\mathrm{discr}}\oplus 0$ in $B$. Then $A+S=B$ and $A\cap S=0$, but  the projection $\pi'_A:B\to A$ to A with kernel $S$ is not continuous: its restriction to $0\oplus V$ is $(0,b)\mapsto (-b,0)$ and is therefore identified with $-\varphi^{-1}$, which is not continuous. As a consequence, the algebraic direct sum $B=A\oplus_{\mathrm{alg}} S$ is not topological. 
\end{example}

\subsection{Hahn-Banach} 

In this section we formulate a general Hahn-Banach theorem and
emphasize, following Beilinson~\cite[\S1.1, Remark]{Beilinson} and
Positselski~\cite{Positselski}, the role played by the countable basis
assumption. 
It is based on the following variant of
Corollary~\ref{cor:Hom-is-surjective}, in which the restrictions are
placed on $C$ rather than on $B$.

\begin{proposition}[{Positselski~\cite[Proposition~3.3]{Positselski}}] 
\label{prop:Hom-is-surjective}
Let $B$ be a linearly topologized vector space and $A\subset B$
a linear subspace, endowed with the induced topology. Let $C$ be a
complete linearly topologized vector space with countable basis. Then
any continuous linear map $A\to C$ can be extended to a continuous
linear map $B\to C$. In other words, the canonical map  
$$
\Hom(B,C)\to \Hom(A,C)
$$
is surjective. \qed
\end{proposition}

Applying this to $C=\bk$ we obtain

\begin{corollary}[Hahn-Banach for linearly topologized vector spaces] \label{cor:general-Hahn-Banach}
Let $B$ be a linearly topologized vector space and $A\subset
B$ a linear subspace, endowed with the induced topology. Then the
canonical map $B^*\to A^*$ is surjective.  
\end{corollary}

We give the direct proof for the reader's convenience, following~\cite{Positselski}. 

\begin{proof}[Proof of Corollary~\ref{cor:general-Hahn-Banach}]
We need to show that any continuous linear map $A\to\bk$ can be
extended to a continuous linear map $B\to \bk$. Let $f:A\to\bk$ be
linear continuous, so that $U=\ker f\subset A$ is open. By definition
of the induced topology on $A$, there exists an open linear subspace
$V\subset B$ such that $U=V\cap A$. The linear map $f$ factorizes
through the surjection $A\to A/U$, so we have a map of discrete vector
spaces $\bar f:A/U\to \bk$. The vector space $A/U$ is a subspace of
the discrete vector space $B/V$, so the map $\bar f$ can be extended
to a (automatically continuous) linear map $\bar g:B/V\to\bk$. The
composition $B\to B/V\to \bk$ provides a continuous linear extension
of the map $f$.  
\end{proof}

Another corollary of Proposition~\ref{prop:Hom-is-surjective} is the
following strengthening of Corollary~\ref{cor:exists-closed-complement},
in which the restrictions are placed on the subspace $A$ rather than
on the ambient space $B$.

\begin{corollary}[{Positselski~\cite[Corollary~3.4]{Positselski}}] 
\label{cor:splitting-Positselski}
Let $B$ be a linearly topologized vector space and $A\subset B$
a linear subspace, endowed with the induced topology. Assume that $A$
is complete and has countable basis.
Then $A$ has a topological complement $S\subset B$ such that
  $B\simeq A\oplus S$. 
\end{corollary}

\begin{proof}
Applying Proposition~\ref{prop:Hom-is-surjective} to the identity
$\Id_A:A\to A=C$, we find a continuous linear map $\pi:B\to A$ with
$\pi|_A=\Id_A$. Then $B\simeq A\oplus\ker\pi$ by
Definition-Lemma~\ref{defi:equivalence-splitting}.  
\end{proof}

\bibliographystyle{abbrv}
\bibliography{000_SHpair}

\end{document}